\documentclass[11pt]{amsart}
\usepackage{amsmath}
\usepackage{amssymb}
\usepackage{amsfonts}
\usepackage{amsthm}
\usepackage{enumerate}
\usepackage[mathscr]{eucal}
\usepackage{graphicx}
\usepackage[dvipdfm, bookmarksnumbered, bookmarksopen, colorlinks, citecolor=blue, linkcolor=blue]{hyperref}

\newtheorem{theorem}{Theorem}[section]
\newtheorem{lemma}[theorem]{Lemma}
\newtheorem{corollary}[theorem]{Corollary}

\theoremstyle{definition}

\newtheorem{prop}[theorem]{Proposition}

\theoremstyle{remark}
\newtheorem{remark}[theorem]{Remark}

\numberwithin{equation}{section}

\def\intslash{\rlap{\kern  .32em $\mspace {.5mu}\backslash$ }\int}
\def\qsl{{\rlap{\kern  .32em $\mspace {.5mu}\backslash$ }\int_{Q_x}}}

\def\R{{\mathbb R}}
\def\Z{{\mathbb Z}}
\def\M{{\mathcal M}}
\def\C{{\mathcal C}}
\def\F{{\mathcal F}}

\def\S{\mathbf S}

\def\N{\mathbb N}

\def\emph#1{{\it #1 }}

\def\pari{\partial}

\def\eg{{\it e.g. }}

\def\supp{{\text{\rm supp}}}

\def\inn#1#2{\langle#1,#2\rangle}

\def\rta{\rightarrow}

\def\card{\text{\rm card}}
\def\lc{\lesssim}

\def\pv{\text{\rm p.v.}}
\def\alp{\alpha}

\def\del{\delta}             
\def\eps{\varepsilon}

\def\tet{\theta}

\def\lam{\lambda}            \def\Lam{\Lambda}

\def\si{\sigma}

              \def\Om{\Omega}
\def\fr{\frac}

\newcommand{\Be}{\begin{equation}}
\newcommand{\Ee}{\end{equation}}
\newcommand{\Bes}{\begin{equation*}}
\newcommand{\Ees}{\end{equation*}}
\newcommand{\Bsp}{\begin{split}}
\newcommand{\Esp}{\end{split}}
\newcommand{\Bm}{\begin{multline}}
\newcommand{\Em}{\end{multline}}
\newcommand{\Bea}{\begin{eqnarray}}
\newcommand{\Eea}{\end{eqnarray}}
\newcommand{\Beas}{\begin{eqnarray*}}
\newcommand{\Eeas}{\end{eqnarray*}}
\newcommand{\Benu}{\begin{enumerate}}
\newcommand{\Eenu}{\end{enumerate}}
\newcommand{\Bi}{\begin{itemize}}
\newcommand{\Ei}{\end{itemize}}
\begin{document}

\title[Multilinear estimates for commutator]{Multilinear estimates for Calder\'on commutators}

\author{Xudong Lai}
\address{Xudong Lai: Institute for Advanced Study in Mathematics, Harbin Institute of Technology, Harbin, 150001, People's Republic of China}
\email{xudonglai@hit.edu.cn\ xudonglai@mail.bnu.edu.cn}
\thanks{The work is supported by China Postdoctoral Science Foundation (No. 2017M621253, No. 2018T110279) and the Fundamental Research Funds for the Central Universities.}


\date{\today}



\begin{abstract}
In this paper, we investigate the multilinear boundedness properties of the higher ($n$-th) order Calder\'on commutator for dimensions larger than two.
We establish all multilinear endpoint estimates for the target space $L^{\frac{d}{d+n},\infty}(\mathbb{R}^d)$, including that Calder\'on commutator maps the product of Lorentz spaces  $L^{d,1}(\mathbb{R}^d)\times\cdots\times L^{d,1}(\mathbb{R}^d)\times L^1(\mathbb{R}^d)$ to $L^{\frac{d}{d+n},\infty}(\mathbb{R}^d)$,
which is the higher dimensional nontrivial generalization of the endpoint estimate that the $n$-th order Calder\'on commutator maps $L^{1}(\mathbb{R})\times\cdots\times L^{1}(\mathbb{R})\times L^1(\mathbb{R})$ to $L^{\frac{1}{1+n},\infty}(\mathbb{R})$.
When considering the target space $L^{r}(\mathbb{R}^d)$ with $r<\frac{d}{d+n}$, some counterexamples are given to show that these multilinear estimates may not hold.
The method in the present paper seems to have a wide range of applications and it can be applied to establish the similar results for Calder\'on commutator with a rough homogeneous kernel.
\end{abstract}

\maketitle

\section{Introduction}
The study of multilinear Calder\'on-Zygmund operators was initiated by Coifman and Meyer (see \cite{CM75}, \cite{CM78}, \cite{MC97}). One of their motivations is to study the second order Calder\'on commutator (see \cite{CM75}).
Now a fruitful theory has grown around the multilinear Calder\'on-Zygmund operator and there are still many works on going, we refer to see the very nice exposition \cite[Chapter 7]{Gra250} and the references therein.
Despite of the intensive research of the multilinear Calder\'on-Zygmund theory, there are still some open problems related to Calder\'on commutators, the original model of multilinear Calder\'on-Zygmund operators. For example, there are no appropriate multilinear endpoint estimates of the higher order Calder\'on commutator for higher dimensions.

In this paper, we investigate the multilinear boundedness properties of the higher ($n$-th) order Calder\'on commutator for dimensions larger than two.
We establish all multilinear endpoint estimates for the target space $L^{\fr{d}{d+n},\infty}(\R^d)$, in which the endpoint estimates exist on a plane $\fr{1}{q_1}+\cdots+\fr{1}{q_n}+\fr{1}{p}=\fr{d+n}{d}$ with $(\fr{1}{q_1},\cdots,\fr{1}{q_n},\fr{1}{p})\in\R^{n+1}$ intersects $[0,1]^{n+1}$. Specially, these endpoint estimates include that Calder\'on commutator maps the product of Lorentz spaces  $L^{d,1}(\R^d)\times\cdots\times L^{d,1}(\R^d)\times L^1(\R^d)$ to $L^{\fr{d}{d+n},\infty}(\R^d)$,
which is the higher dimensional nontrivial generalization of the endpoint estimate that the $n$-th order Calder\'on commutator maps $L^{1}(\R)\times\cdots\times L^{1}(\R)\times L^1(\R)$ to $L^{\fr{1}{1+n},\infty}(\R)$. If the dimension $d=1$, the above endpoint estimates for the $n$-th order commutators on products of $L^1(\R)$ spaces have been obtained by C. P. Calder\'on \cite{CCal75} when $n=1$, by Coifman and Meyer \cite{CM75} when $n=1,2$ and by Duong, Grafakos and Yan \cite{DGY10} when $n\geq 1$. However when the dimension $d\geq 2$, things become more complicated since Calder\'on commutator in this case is a non standard multilinear Calder\'on-Zygmund operator. No appropriate multilinear Calder\'on-Zygmund theory can be applied to it directly.
Therefore it is interesting to establish the multilinear estimates of Cader\'on commutator for $d\geq2$ and the purpose of the present paper is to develop the theory in this respect.

Before stating our results, we give some notation and the background.
Define the higher ($n$-th) order Calder\'on commutator by
\Be\label{e:12com}
\mathcal{C}[\nabla A_1,\cdots,\nabla A_n,f](x)=\pv \int_{\R^d} K(x-y)\Big(\prod_{i=1}^n\fr{A_{i}(x)-A_i(y)}{|x-y|}\Big)\cdot f(y)dy,
\Ee
where $n$ is a positive integer and $K$ is the Calder\'on-Zygmund convolution kernel on $\mathbb{R}^d\setminus\{0\}\, (d\ge2)$ which means that $K$ satisfies the following three conditions:
\begin{equation}\label{e:12kb}
|K(x)|\lc|x|^{-d},
\end{equation}
\begin{equation}\label{e:12K_2}
\int_{R<|x|<2R}K(x)\Big(\fr{x}{|x|}\Big)^{\alp}dx=0, \ \forall R>0,\ \forall \alp\in\Z_+^d\ \text{with  $|\alp|=n$},
\end{equation}
\begin{equation}\label{e:12kr}
|K(x-y)-K(x)|\lc\frac{|y|^\del}{|x|^{d+\del}}\ \ \text{for some $0<\del\leq1$ if}\ \ |x|>2|y|.
\end{equation}

Such kind of commutator was first introduced by A. P. Calder\'on in \cite{Cal65} for the first order with $K(x)$ a homogeneous kernel
and also later in \cite{Cal77} \cite{Cal78} for the higher order one (see also \cite{CM75}, \cite{CM78}).
It is easy to see that $\mathcal{C}[\nabla A_1,\cdots,\nabla A_n,f](x)$ is well defined for $A_1$, $\cdots$, $A_n$, $f\in C_c^\infty(\mathbb{R}^d)$.
For its applications, let us look at the first order Calder\'on commutator \eqref{e:12com}.
Indeed $\mathcal{C}[\nabla A,f](x)$ is a generalization of
\Be\label{e:12comhilb}
\begin{split}
[A, S]f(x)=A(x)S(f)(x)-S(Af)(x)=
-\pv\fr{1}{\pi}\int_{\R}\fr{1}{x-y}\fr{A(x)-A(y)}{x-y}f(y)dy
\end{split}
\Ee
where $S=\fr{d}{dx}\circ H$ and $H$ denotes the Hilbert transform (one can deduce $S=\fr{d}{dx}\circ H$ just by taking a derivation into the kernel $\fr{1}{\pi x}$ or utilizing the Fourier transform for both sides).
It is well known that the commutator $[A, S]$ is a fundamental operator in harmonic analysis and plays an important role in the theory of the Cauchy integral along Lipschitz curve in $\mathbb{C}$, the boundary value problem of elliptic equation on non-smooth domain, and the Kato square root problem on $\R$ (see \eg \cite{Cal65}, \cite{Cal78}, \cite{Fef74}, \cite{MC97}, \cite{Gra250} for the details). Recently, there has been a renewed interest into the commutator $[A,S]$ and {\it d-commutator} introduced by M. Christ and J. Journ\'e (see \cite{CJ87}) since they have applications in the mixing flow problem (see \eg \cite{SSS15}, \cite{HSSS17}).

In this paper, we are interested in the following strong type multilinear estimate (or weak type estimate)
\Be\label{e:12mmulti}
\|\mathcal{C}[\nabla A_1,\cdots,\nabla A_n,f]\|_{L^r(\R^d)}\lc \Big(\prod_{i=1}^n\|\nabla A_i\|_{L^{q_i}(\R^d)}\Big)\|f\|_{L^p(\R^d)}
\Ee
where $\fr{1}{r}=\big(\sum_{i=1}^n\fr{1}{q_i}\big)+\fr{1}{p}$ with $1\leq q_i\leq\infty$, $(i=1,\cdots,n)$ and $1\leq p\leq\infty$. Our main results are as follows.

\begin{theorem}\label{t:12}
Let $d\geq2$ and $n$ be a positive integer. Suppose $K$ satisfies $(\ref{e:12kb}), (\ref{e:12K_2})$
and $(\ref{e:12kr})$.
Assume that  $\fr{1}{r}=\big(\sum_{i=1}^n\fr{1}{q_i}\big)+\fr{1}{p}$ with $1\leq q_i\leq\infty$ $(i=1,\cdots,n)$,  and $1\leq p\leq\infty$. We have the following conclusions:

{\rm (i).} If $\fr{d}{d+n}< r<\infty$, $1<q_i\leq\infty$ $(i=1,\cdots,n)$ and $1<p\leq\infty$, then the multilinear estimate \eqref{e:12mmulti} holds.

{\rm (ii).} If $\fr{d}{d+n}\leq r<\infty$ with $q_i=1$ for some $i=1,\cdots,n$; or $p=1$; or $r=\fr{d}{d+n}$, then there exists a constant $C>0$ such that
\Be\label{e:12muweak}
\|\mathcal{C}[\nabla A_1,\cdots,\nabla A_n,f]\|_{L^{r,\infty}(\R^d)}\leq C \Big(\prod_{i=1}^n\|\nabla A_i\|_{L^{q_i}(\R^d)}\Big)\|f\|_{L^p(\R^d)}
\Ee
and in this case, if $q_i=d$ for some $i=1,\cdots,n$, $L^{q_i}(\R^d)$ in the above inequality should be replaced by $L^{d,1}(\R^d)$, the standard Lorentz space. Specially, we have the following endpoint estimate
\Be\label{e:12mulendpoint}
\|\mathcal{C}[\nabla A_1,\cdots,\nabla A_n,f]\|_{L^{\fr{d}{d+n},\infty}(\R^d)}\leq C \Big(\prod_{i=1}^n\|\nabla A_i\|_{L^{d,1}(\R^d)}\Big)\|f\|_{L^1(\R^d)}.
\Ee

{\rm(iii).} If $0<r<\fr{d}{d+n}$, $1\leq q_i\leq\infty$ $(i=1,\cdots,n)$ and $1\leq p\leq\infty$, there exist functions $A_i$ for $i=1,\cdots,n$, and $f$ such that $\|\nabla A_i\|_{L^{q_i}(\R^d)}<\infty$ for $i=1,\cdots,n$, and $\|f\|_{L^p(\R^d)}<\infty$. But
\Bes
\begin{split}
\mathcal{C}[\nabla A_1,\cdots,\nabla A_n,f](x)=\infty\ \ \text{in a ball in $\R^d$.}
\end{split}
\Ees
\end{theorem}

\begin{remark}
Notice that (i) gives strong type estimates \eqref{e:12mmulti} for $\fr{d}{d+n}< r<\infty$. (ii) gives all endpoint estimates for $\fr{d}{d+n}\leq r\leq1$, especially the case $r=\fr{d}{d+n}$ where the endpoints $\big(\fr{1}{q_1},\cdots,\fr{1}{q_n},\fr{1}{p}\big)$ exist in the intersection between the plane $\big(\sum_{i=1}^n\fr{1}{q_i}\big)+\fr{1}{p}=\fr{d+n}{d}$ and $[0,1]^{n+1}$, which is the most difficult part in our proof.
Here we point out that the condition $r\geq\fr{d}{d+n}$ is crucial in the proof of (i) and (ii) in Theorem \ref{t:12}, which will be emphasized further in the proof where we use this condition. Our basic strategy is first to show \eqref{e:12mmulti} for $1\leq r<\infty$ and (ii), then use the multilinear interpolation between \eqref{e:12mmulti} for $1\leq r<\infty$ and the result of (ii), to justify the rest part of \eqref{e:12mmulti} for $\fr{d}{d+n}<r<1$. All those will be clear in our proof.
Obviously, the conclusion of (iii) indicates that the requirement $r\geq\fr{d}{d+n}$ is a necessary condition to guarantee the strong type estimates (or weak type estimates) \eqref{e:12mmulti} hold,
thus our results in Theorem \ref{t:12} are optimal in this sense.
Some counterexamples will be constructed to prove conclusion (iii).
\end{remark}

\begin{remark}
Notice that $L^{1,1}(\R)=L^1(\R)$. Therefore when the dimension $d=1$, \eqref{e:12mulendpoint} turns out to be the $n$-th Calder\'on commutator mapping $L^{1}(\R)\times\cdots\times L^{1}(\R)\times L^1(\R)$ to $L^{\fr{1}{1+n},\infty}(\R)$, which has been previously proved by Duong, Grafakos and Yan \cite{DGY10}. To the best knowledge of the author, \eqref{e:12mulendpoint} is new when $d\geq 2$. Currently we still do not know whether $L^{d,1}(\R^d)$ in \eqref{e:12mulendpoint} could be replaced by $L^{d,1+\eps}(\R^d)$ for some $\eps>0$ when $d\geq 2$ and we will further explore this problem in our future
research.
\end{remark}

We next briefly introduce the methods employed and the main procedures in the proof of Theorem \ref{t:12}. We first establish the assertion (i) of Theorem \ref{t:12} in the case $1\leq r<\infty$ based on the recent deep result of A. Seeger, C. K. Smart and B. Street in \cite{SSS15}.
Next we show that if $q_i=\infty$ with $i=1,\cdots,n$ and $p=1$, i.e. $A_i$ is a Lipschitz function, then the weak type $L^{1,\infty}(\R^d)$ boundedness holds by the standard Calder\'on-Zygmund theory. We will devote to proving (ii), i.e. we need to give a weak type estimate. In the case of (ii), by our condition, $A_i$ satisfies $\nabla A_i\in L^{q_i}(\R^d)$. We will construct an {\it exceptional set\/} which satisfies the required weak type estimate. And on the complementary set of {\it exceptional set\/}, the function $A_i$ is a Lipschitz function. Then, roughly speaking, the strong type estimate in (i) and the weak type $L^{1,\infty}(\R^d)$ boundedness of $\mathcal{C}[\nabla A_1,\cdots,\nabla A_n, f](x)$ could be applied on the complementary set of {\it exceptional set\/}.
The idea partly comes from C. P. Calder\'on \cite{CCal75}, \cite{CCal79}. However we develop further more here.
Our argument works once we establish the strong type estimate \eqref{e:12mmulti} when $1<r<\infty$, $1<q_1,\cdots, q_n\leq\infty$, $1\leq p\leq\infty$ and weak type $L^{1,\infty}(\R^d)$ boundedness when $r=1$, $q_1=\cdots=q_n=\infty$, $p=1$.

The strategy to construct the {\it exceptional set\/} is as follows. Notice that the estimate $\|\nabla A\|_{L^q(\R^d)}$ is related to the Sobolev space $W^{1,q}(\R^d)$. When $1\leq q<d$, it is well known that Sobolev space $W^{1,q}(\R^d)$ is embedded into $L^{q*}(\R^d)$ with $\fr{1}{q^*}=\fr{1}{q}-\fr{1}{d}$. This property is crucial to help us establish a boundedness property of maximal operator (see Lemma \ref{l:12qleqd}). When $q>d$, {\it exceptional set\/} can be constructed by using the Mary Weiss maximal operator $\M$ (see Subsection \ref{s:1221} for its definition), which maps $L^q(\R^d)$ to $L^q(\R^d)$ (or $L^{q,\infty}(\R^d)$) only when $q>d$. But when $q=d$, the critical Sobolev $W^{1,d}(\R^d)$ is imbedded into an Orlicz space (see \cite{AF03}) which may be not useful to us.
This forces us to study the Mary Weiss maximal operator on $L^d(\R^d)$, which is quite challenging. Fortunately, we find a substitute that $\M$ maps the Lorentz space $L^{d,1}(\R^d)$ to $L^{d,\infty}(\R^d)$ which is enough to construct an {\it exceptional set\/}.
Base on this, we can establish the multilinear endpoint estimate that $\mathcal{C}[\nabla A_1,\cdots,\nabla A_n, f](x)$ maps $L^{d,1}(\R^d)\times\cdots\times L^{d,1}(\R^d)\times L^1(\R^d)$ to $L^{\fr{d}{d+n},\infty}(\R^d)$. Although we assume that $d\geq 2$ in our main results, the proof presented in this paper is also valid for $d=1$.
Therefore even when $d=1$, the proof of \eqref{e:12mulendpoint}  here is quite different from that by Duong, Grafakos and Yan \cite{DGY10}, thus we give a new proof of \eqref{e:12mulendpoint} for $d=1$.

As aforementioned, the above method built in this paper works as long as we establish the strong type estimate \eqref{e:12mmulti} when $1<r<\infty$ and weak type $L^{1,\infty}(\R^d)$ boundedness when $r=1$, $q_1=\cdots=q_n=\infty$, $p=1$.
Therefore we can use the method here to establish the similar multilinear estimates of Calder\'on commutator with a homogeneous rough kernel. Define the higher order Calder\'on commutator with a rough kernel by
\Bes
\mathcal{C}_\Om[\nabla A_1,\cdots,\nabla A_n,f](x)=\pv \int_{\R^d} \fr{\Om(x-y)}{|x-y|^d}\Big(\prod_{i=1}^n\fr{A_{i}(x)-A_i(y)}{|x-y|}\Big)\cdot f(y)dy,
\Ees
here $\Om$ is a function defined on $\R^d\setminus\{0\}$ which satisfies:
\begin{equation}\label{e:12Omho}
\Om(r\tet)=\Om(\tet)\ \ \text{for $r>0$, $\tet\in\mathbf{S}^{d-1}$;}\ \
\Om(-\tet)=(-1)^{n+1}\Om(\tet)\footnote{One may also consider the case $\Om(-\tet)=(-1)^{n}\Om(\tet)$ with some other moment cancelation conditions, we refer to see Remark \ref{r:12cancelation} for further discussion.}
\end{equation}
and $\Om\in L^1(\mathbf{S}^{d-1})$. $\mathbf{S}^{d-1}$ is the unit sphere in $\R^d$. Similar to those in Theorem \ref{t:12}, we have the following result.
\begin{theorem}\label{t:12r}
Suppose $\Om$ satisfies $(\ref{e:12Omho})$
and $\Om\in L\log^+L(\S^{d-1})$ for $d\geq2$.
Then all the results in Theorem \ref{t:12} also hold for $\mathcal{C}_\Om[\nabla A_1,\cdots,\nabla A_n,f](x).$
\end{theorem}

When $n=1$, part of results in Theorem \ref{t:12r} have been established by A. P. Calder\'on \cite{Cal65} and C. P. Calder\'on \cite{CCal75} \cite{CCal79}. We summarize their results \cite{Cal65}, \cite{CCal75}, \cite{CCal79} in Figure \ref{f:12}.
More precisely, A. P. Calder\'on \cite{Cal65} showed that if $\fr{1}{r}=\fr{1}{q}+\fr{1}{p}$ with $1<r<\infty$, $1<q\leq\infty$, $1<p<\infty$, then \eqref{e:12mmulti} holds when $\Omega\in L\log^+L(\S^{d-1})$ (see the region with diagonal lines in Figure \ref{f:12}).
Later C. P. Calder\'on \cite{CCal75} extended these results to the boundary of the region with diagonal lines where he proved \eqref{e:12mmulti} is still true in the case $1<r=q<\infty$, $p=\infty$ and in the case $r=1, q>1, p>1$.
C. P. Calder\'on \cite{CCal75} also showed that if $\Om$ satisfies the H\"ormander condition,  then \eqref{e:12mmulti} holds when ${d}/(d+1)< r<1$, $q>d$, $p>1$ (see the region with vertical lines in Figure \ref{f:12}).
In \cite{CCal79}, C. P. Calder\'on showed that if $\fr{d}{d+1}\leq r\leq1$, $1\leq q<d$, $1< p\leq\infty$, then the weak type estimate \eqref{e:12muweak} holds when $\Om\in L\log^+L(\S^{d-1})$ (see the region with horizontal lines in Figure \ref{f:12}). With the above results in hand, by using the interpolation arguments, one may easily get the strong type estimate \eqref{e:12mmulti} holds for $\fr{d}{d+n}<r<1$, $1<q<\infty$, $1<p<\infty$ if  $\Om\in L\log^+L(\S^{d-1})$. Recently Fong \cite{Fon16} considered the special case $\Om\equiv1$ and used the time-frequency analysis method to show \eqref{e:12mmulti} holds for $\fr{d}{d+n}<r<\infty$, $1<q<\infty$, $1<p<\infty$.

\begin{figure}[!h]\label{f:12}
\centering
\includegraphics[height=0.38\textwidth]{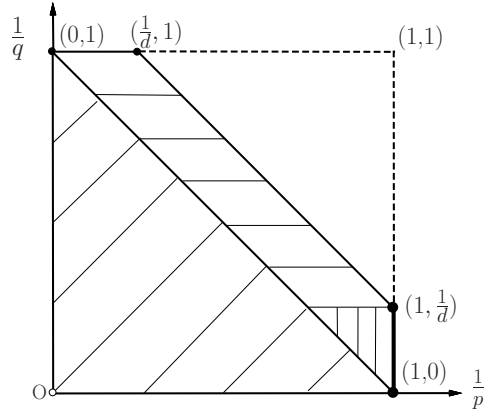}
\caption{\small{In the case $n=1$, our main results in Theorem \ref{t:12r} are new when $0<\fr{1}{q}\leq \fr{1}{d}$ and $\fr{1}{p}=1$, see the \textbf{bold} line including the endpoint $(\frac{1}{p},\frac{1}{q})=(1,\frac{1}{d})$.}}
\end{figure}

For the endpoint $(\frac{1}{p},\frac{1}{q})=(1,0)$, the weak type $L^{1,\infty}(\R^d)$ boundedness of $\C_\Om[\nabla A,f]$ with $\Om\in L\log^+L(\S^{d-1})$ has been recently derived by Ding and the author \cite{DL15a}. The contribution of Theorem \ref{t:12r} in the case $n=1$ is the estimates with $p=1$, $0<\fr{1}{q}\leq \fr{1}{d}$ (see the bold line including the endpoint $(\frac{1}{p},\frac{1}{q})=(1,\frac{1}{d})$ in Figure \ref{f:12}), which complements the aforementioned works for $n=1$.  To the best knowledge of the author, Theorem \ref{t:12r} is new when $n\geq 2$.

An immediate consequence of Theorem \ref{t:12} or \ref{t:12r} is the following $n$-th order commutator of the Riesz transform with $n$-th derivation which may have potential applications in partial differential equations.

\begin{corollary}\label{c:12c}
Let $R_j$ be the Riesz transform. Then all the results in Theorem \ref{t:12} also hold for the following operator
\Bes
\begin{split}
[A_1,\cdots,[A_n,&\pari^{\alp}\circ R_j]\cdots]f(x)\\
&=\pv\int_{\R^d}\partial^{\alp}_x\Big(\fr{x_j-y_j}{|x-y|^{d+1}}\Big)\cdot\Big(\prod_{i=1}^n{[A_{i}(x)-A_i(y)]}\Big)\cdot f(y)dy
\end{split}
\Ees
where $\alp\in\Z_+^d$ is a multi-indice with $|\alp|=n$.
\end{corollary}

This paper is organized as follows.
In Section \ref{s:122}, we give the proof of Theorem \ref{t:12}, which will be divided into several cases. First some preliminary lemmas are presented in Subsection \ref{s:1221}. Subsection \ref{s:1222} is devoted to proving (i) of Theorem \ref{t:12} in the case $1\leq r<\infty$ and weak type $L^{1,\infty}(\R^d)$ boundedness on $L^\infty(\R^d)\times\cdots\times L^\infty(\R^d)\times L^1(\R^d)$. The proofs of (ii) in Theorem \ref{t:12} are given in Subsections \ref{s:1223}, \ref{s:1224} and \ref{s:1225}. In Subsection \ref{s:1226}, we proceed to proving the rest part of (i) in Theorem \ref{t:12} by the multilinear interpolation theorem. Finally some counterexamples are given in Subsection \ref{s:1227} to prove (iii) in Theorem \ref{t:12}. The proof of Theorem \ref{t:12r} is similar to that of Theorem \ref{t:12}. So in Section \ref{s:123}, we outline the proof of Theorem \ref{t:12r}.
\vskip0.24cm
\textbf{Notation}. Throughout this paper, we only consider the dimension $d\ge2$ and the letter $C$ stands for a positive finite constant which is independent of the essential variables and not necessarily the same one in each occurrence. $A\lc B$ means $A\leq CB$ for some constant $C$. By the notation $C_\eps$ means that the constant depends on the parameter $\eps$. $A\approx B$ means that $A\lc B$ and $B\lc A$.
$n$ represents the order of Calder\'on commutator. The indices $r$, $q_1,\cdots, q_n$ and $p$ satisfy $\fr{1}{r}=\big(\sum_{i=1}^n\fr{1}{q_i}\big)+\fr{1}{p}$ with $1\leq q_i\leq\infty$ $(i=1,\cdots,n)$ and $1\leq p\leq\infty$ in the whole paper. For a set $E\subset\R^d$, we denote by  $|E|$ or $m(E)$ the Lebesgue measure of $E$. $\mathbf{S}^{d-1}$ is the unit sphere in $\R^d$. $d\si$ denotes the spherical measure on $\mathbf{S}^{d-1}$.  $\nabla A$ will stand for the vector $(\pari_1A,\cdots,\pari_dA)$ where $\pari_i A(x)=\pari A(x)/\pari x_i$.
Define $$\|\nabla A\|_{X}=\Big\|\Big(\sum_{i=1}^d|\pari_iA|^2\Big)^{\fr{1}{2}}\Big\|_{X}$$
for $X=L^p(\R^d)$ or $X=L^{d,1}(\R^d)$. $\Z_+$ denotes the set of all nonnegative integers and $\Z_+^d=\underbrace{\Z_+\times \cdots\times \Z_+}_d.$
\vskip1cm

\section {Proof of Theorem \ref{t:12}}\label{s:122}
\subsection{Some preliminary lemmas}\label{s:1221}\quad
\vskip0.2cm
Before giving the proof of Theorem \ref{t:12}, we introduce some lemmas which play a key role in the proof of Theorem \ref{t:12}.
For those readers who are not familiar with the theory of the Lorentz space $L^{p,q}(\R^d)$, we refer to see \cite[Chapter V.3]{Ste71}. We will use the theory of the Lorentz space $L^{p,q}(\R^d)$ in Lemma \ref{l:11md}.
Now we begin by some properties of a special maximal function which was introduced by Mary Weiss (see \cite{CCal75}). It is defined as
$$\M(\nabla A)(x)=\sup_{h\in\R^d\setminus\{0\}}\fr{|A(x+h)-A(x)|}{|h|}.$$
\begin{lemma}\label{l:mw}
Let $\nabla A\in L^p(\R^d)$ with $p>d$. Then $\M$ is bounded on $L^p(\R^d)$, that is
$$\|\M(\nabla A)\|_{L^p(\R^d)}\leq C\|\nabla A\|_{L^p(\R^d)},$$
where the constant $C$ is independent of $A$.
\end{lemma}
\begin{proof}
By using a standard limiting argument, we only need to consider $A$ as a $C^\infty$ function with compact support. Then the lemma just follows from the inequality
$$\fr{|A(x)-A(y)|}{|x-y|}\lc\Big(\fr{1}{|x-y|^d}\int_{|x-z|\leq 2|x-y|}|\nabla A(z)|^{q}dz\Big)^{\fr{1}{q}},$$
which holds for any $q>d$ (see \cite[Lemma 1.4]{CCal75}) and the fact that the Hardy-Littlewood maximal operator is of strong type $(p,p)$ for $p>1$.
\end{proof}
\begin{lemma}\label{l:11md}
Let $\nabla A\in L^{d,1}(\R^d)$, the standard Lorentz space. Then for any $\lam>0$, there exist a finite constant $C$ independent of $A$ such that
$$\lam^d|\{x\in\R^d:\M(\nabla A)(x)>\lam\}|\leq C\|\nabla A\|^d_{L^{d,1}(\R^d)}.$$
\end{lemma}
\begin{proof}
It suffices to consider $A$ as a smooth function with compact support. By the formula given in \cite[page 125, (17)]{Ste70}, we may write
\Bes
A(x)=C_d\sum_{i=1}^d\int_{\R^d}\fr{x_i-y_i}{|x-y|^d}\pari_i A(y)dy=\mathbb{K}*f(x)
\Ees
where $\mathbb{K}(x)=1/|x|^{d-1}$, $f=C_d\sum_{j=1}^dR_j(\pari_j A)$ with $R_j$ the Riesz transforms. By using the fact the Riesz transform $R_j$ maps $L^{d,1}(\R^d)$ to itself which follows from the general form of the Marcinkiewicz interpolation theorem (see \cite[Theorem 3.15 in page 197]{Ste71}), one can easily get that
$$\|f\|_{L^{d,1}(\R^d)}\lc\|\nabla A\|_{L^{d,1}(\R^d)}.$$
Hence to prove the lemma, it is enough to show that
\Be\label{e:11weissmaximal}
\lam^d|\{x\in\R^d:\M(\nabla A)(x)>\lam\}|\lc\|f\|^d_{L^{d,1}(\R^d)}
\Ee
with $A=\mathbb{K}*f$. In the following our goal is to prove that for any $x\in\R^d$, the estimate
$$|A(x+h)-A(x)|\lc |h|T(f)(x)$$
holds uniformly for $h\in\R^d\setminus\{0\}$ with $T$ an operator maps $L^{d,1}(\R^d)$ to $L^{d,\infty}(\R^d)$. Once we prove this, we get \eqref{e:11weissmaximal} and hence complete the proof of Lemma \ref{l:11md}. We write
\Bes
\begin{split}
A(&x+h)-A(x)\\
&=\int_{|x-y|\leq2|h|}|x+h-y|^{-d+1}f(y)dy-\int_{|x-y|\leq2|h|}|x-y|^{-d+1}f(y)dy\\
&\ \ \ \ +\int_{|x-y|>2|h|}\Big(|x+h-y|^{-d+1}-|x-y|^{-d+1}\Big)f(y)dy\\
&=I+II+III.
\end{split}
\Ees

Let us first consider $I$. By an elementary calculation, one may get $\mathbb{K}\in L^{d',\infty}(\R^d)$ where $d'=d/(d-1)$. Set $B(x,r)=\{y\in\R^d:|x-y|\leq r\}$. Using the rearrangement inequality (see \cite[page 74, Exercise 1.4.1]{Gra249}), we have
\Bes
\begin{split}
|I|&\leq\int_{\R^d}\mathbb{K}(x+h-y)|f\chi_{B(x,2|h|)}(y)|dy\leq\int_0^\infty \mathbb{K}^*(s)(f\chi_{B(x,2|h|)})^*(s)ds\\
&\leq\Big(\int_0^\infty (f\chi_{B(x,2|h|)})^*(s)s^{\fr{1}{d}}\fr{ds}{s}\Big)\cdot\sup_{s>0}\Big(\mathbb{K}^*(s)s^{\fr{1}{d'}}\Big)\\
&\lc\|f\chi_{B(x,2|h|)}\|_{L^{d,1}(\R^d)}\|\mathbb{K}\|_{L^{d',\infty}(\R^d)},
\end{split}
\Ees
here $f^*$ represents the decreasing rearrangement of $f$.
Using the definition of Lorentz space, one may get $\|\chi_E\|_{L^{d,1}(\R^d)}=\|\chi_E\|_{L^d(\R^d)}$ holds for any characteristic function $\chi_E$ of set $E$ of finite Lebesgue measure, thus $\|\chi_{B(x,2|h|)}\|_{L^{d,1}(\R^d)}=C_d|h|$. Therefore we get
$$|I|\lc |h|\Lam(f)(x),\ \ \text{where}\ \ \Lam(f)(x)=\sup_{r>0}\fr{\|f\chi_{B(x,r)}\|_{L^{d,1}(\R^d)}}{\|\chi_{B(x,r)}\|_{L^{d,1}(\R^d)}}.$$
Below we need to show that the operator $\Lam$ maps $L^{d,1}(\R^d)$ to $L^{d,\infty}(\R^d)$, which can be found in \cite{Ste81}. Since the proof is short, for completeness, we also give a proof here.  Note that $L^{d,1}(\R^d)$ is a Banach space (see \cite[page 204, Theorem 3.22]{Ste71}), it is sufficient to show that $\Lam$ maps the characteristic function $\chi_E\in L^{d,1}(\R^d)$ to $L^{d,\infty}(\R^d)$ (see \cite[page 62, Lemma 1.4.20]{Gra249}). However
in this case, it is equivalent to show that
$$\lam|\{x\in\R^d: M(\chi_E)(x)>\lam\}|\lc\|\chi_E\|_{L^1(\R^d)},$$
where $M$ is the Hardy-Littlewood maximal operator. It is well known that $M$ is of weak type (1,1), hence we have shown that $\Lam$ maps $L^{d,1}(\R^d)$ to $L^{d,\infty}(\R^d)$.

Next we consider $II$. This estimate is quite simple. Since the kernel $k(x)=\eps^{-1}|x|^{-d+1}\chi_{\{|x|\leq \eps\}}$ is a radial non-increasing function and $L^1$ integrable in $\R^d$, we get
$$|II|\lc \|k\|_{L^1(\R^d)}|h|M(f)(x).$$
Notice that $L^{p,1}(\R^d)\subset L^p(\R^d)$ and $M$ is of strong type $(p,p)$, $1<p<\infty$, of course those imply that $M$ maps $L^{d,1}(\R^d)$ to $L^{d,\infty}(\R^d)$.

Finally we give an estimate of $III$. Notice that we only consider $|x-y|>2|h|$. Then by the Taylor expansion of $|x-y+h|^{-d+1}$, one may have
\Be\label{e:11taylor}
\fr{1}{|x-y+h|^{d-1}}-\fr{1}{|x-y|^{d-1}}=(-d+1)\sum_{j=1}^dh_j\fr{x_j-y_j}{|x-y|^{d+1}}+R(x,y,h)
\Ee
where the Taylor expansion's remainder term $R(x,y,h)$ satisfies
$$|R(x,y,h)|\leq C|h|^2|x-y|^{-d-1}\ \text{if } |x-y|>2|h|.$$
Inserting \eqref{e:11taylor} into the term $III$ with the above estimate of $R(x,y,h)$, we conclude that
$$|III|\lc |h|\sum_{j=1}^d R_j^*(f)(x)+|h|^2\int_{|x-y|>2|h|}|x-y|^{-d-1}|f(y)|dy$$
where $R_j^*$ is the maximal Riesz transform which is defined by
$$R_j^*(f)(x)=\sup_{\eps>0}\Big|\int_{|x-y|>\eps}\fr{x_j-y_j}{|x-y|^{d+1}}f(y)dy\Big|.$$
Since $R_j^*$ is bounded on $L^p(\R^d)$, $1<p<\infty$, one immediately gets that $R_j^*$ maps $L^{d,1}(\R^d)$ to $L^{d,\infty}(\R^d)$.
The second term which controls $III$ can be dealt with the same way as we do in the estimate of $II$ once we notice that the function $\eps|x|^{-d-1}\chi_{\{|x|>\eps\}}$ is radial non-increasing and $L^1$ integrable.
\end{proof}
\begin{remark}
Here it should be pointed out that some idea in this proof is similar to that in \cite{Ste81}, where E. M. Stein proved that for a function $F$ defined in $\R^d$ with $\nabla F\in L^{d,1}_{loc}(\R^d)$, then $F$ is equivalent with a continuous function and
\Be\label{e:11smooth}
F(x+h)-F(x)-h(\nabla F)(x)=o(|h|)\ \text{for almost every $x$,}
\Ee
as $|h|\rta0$. The method of proving \eqref{e:11smooth} in \cite{Ste81} is  just giving a direct estimate of \eqref{e:11smooth}.  See also another proof by using elementary principle in \cite{DS84}, \cite{DSm84}. The property of the maximal operator $\M$ that maps $L^{d,1}(\R^d)$ to $L^{d,\infty}(\R^d)$ seems to be more powerful since it implies \eqref{e:11smooth} immediately. In fact, using the dense limiting arguments and Lemma \ref{l:11md}, we get for any function $F$ defined in $\R^d$ with $\nabla F\in L^{d,1}(\R^d)$,
$$\lim_{s\rta0}\fr{F(x+s\tet)-F(x)}{s}=(\nabla F)(x)\cdot \tet,\ \text{for any $\tet\in\S^{d-1}$, a.e. $x\in\R^d$,}$$
which is inequivalent to \eqref{e:11smooth}.
\end{remark}

\begin{lemma}\label{l:12qleqd}
Let $\nabla A\in L^p(\R^d)$ with $1\leq p<d$. Set $1/s=1/p-1/d$.
Define the maximal operator $\mathfrak{M}_{s}$ and the Hardy-Littlewood maximal operator of order $p$ $M_{p}$ by
\Bes
\begin{split}
\mathfrak{M}_{s}(\nabla A)(x)&=\sup_{r>0}\Big(\fr{1}{|Q(x,r)|}\int_{Q(x,r)}\Big|\fr{A(x)-A(y)}{r}\Big|^{s}dy\Big)^{1/s},\\
M_p(f)(x)&=\sup\limits_{r>0}\bigg(\frac 1{|Q(x,r)|}\int_{Q(x,r)}|f(y)|^pdy\bigg)^{1/p},
\end{split}
\Ees
where $Q(x,r)$ is a cube with center $x$ and sidelength $r$. Then we have
\Bes
\mathfrak{M}_{s}(\nabla A)(x)\lc M_p(\nabla A)(x).
\Ees
\end{lemma}
\begin{proof}
We refer to see \cite[Lemma 3.2]{CCal79} and its proof there from line (3.2.2) to (3.2.7).
\end{proof}
\begin{lemma}\label{l:12disq1infty}
Let $\{Q_k\}_{k}$ be the disjoint cubes in $\R^d$. Denote by $l(Q_k)$ the side length of $Q_k$. Suppose $\Om$ satisfies \eqref{e:12Omho}. Define the operator $T_s$ as
$$T_{s}(f)(x)=\sum_k\int_{Q_k}\fr{|\Om(x-y)|\cdot l(Q_k)^{s}}{[l(Q_k)+|x-y|]^{d+s}}|f(y)|dy.$$
Then for any $s>0$ and $1\leq q\leq\infty$, we get that,
\Bes
\|T_{s}(f)\|_{L^q(\R^d)}\lc\|\Om\|_{L^1(\S^{d-1})}\|f\|_{L^q(\R^d)}.
\Ees
\end{lemma}
\begin{proof}
If $q=1$, Lemma \ref{l:12disq1infty} just follows from the Fubini theorem. In fact, we have
\Bes
\|T_s(f)\|_{L^1(\R^d)}\leq\sum_{Q_k}\int_{Q_k}\Big|\int_{\R^d}\fr{|\Om(x-y)|\cdot l(Q_k)^{s}}{[l(Q_k)+|x-y|]^{d+s}}dx\Big|\cdot|f(y)|dy\lc\|\Om\|_{L^1(\S^{d-1})}\|f\|_{L^1(\R^d)},
\Ees
here we use that $Q_k$s are cubes disjoint each other. If $q=\infty$, applying the Fubini theorem again,
\Bes
|T_s(f)(x)|\leq\sum_{Q_k}\|f\|_{L^\infty(Q_k)}\sup_{x\in\R^d}\int_{Q_k}\fr{|\Om(x-y)|\cdot l(Q_k)^{s}}{[l(Q_k)+|x-y|]^{d+s}}dy\lc\|\Om\|_{L^1(\S^{d-1})}\|f\|_{L^\infty(\R^d)}.
\Ees
Now using the Marcinkiewicz interpolation theorem (see \eg \cite{Ste71}), one may get $T_s$ maps $L^q(\R^d)$ to $L^q(\R^d)$ for any $1<q<\infty$. Hence we complete the proof.
\end{proof}

In the following, we begin to give the proof of Theorem \ref{t:12}. We will first show our theorem for $r\geq1$ which is not quite complicated. Define the multi-indice set
\Bes
{\mathrm{MI}}=\Big\{(\fr{1}{q_1},\cdots,\fr{1}{q_n},\fr{1}{p}):\ \fr{1}{r}=\big(\sum_{i=1}^n\fr{1}{q_i}\big)+\fr{1}{p}, \fr{d}{d+n}\leq r<\infty, 1\leq q_1,\cdots,q_n,p\leq\infty\Big\}.
\Ees
If $\fr{d}{d+n}\leq r\leq 1$, we will divide the proof into several cases according whether $q_i$ is bigger than $d$ or smaller than $d$. And in this case, we will establish the weak type estimate at all boundary points of $\mathrm{MI}$. Although we don't take a rigorous classification, we will cover all cases for $\fr{d}{d+n}\leq r<\infty$, $1\leq q_1,\cdots,q_n\leq\infty$ and $1\leq p\leq\infty$.
Next we will use the multilinear interpolation to establish the strong type estimate in the interior of $\mathrm{MI}$ between $r>1$ and $r=\fr{d}{d+n}$. Finally, we give some examples to show that if $0<r<\fr{d}{d+n}$, there are no multilinear strong type estimates like \eqref{e:12mmulti} (or weak type estimates).
\vskip0.24cm

\subsection{Case: $1\leq r<\infty$}\label{s:1222}\quad
\vskip0.24cm
\begin{prop}\label{p:12strongr}
Let $1\leq r<+\infty$, $1<q_i\leq\infty$, $i=1,\cdots,n$, $1<p\leq\infty$. Then the strong type estimate \eqref{e:12mmulti} holds.
\end{prop}
\begin{proof}
We do not plan to give a direct proof here. The proof relies on the recent deep results in \cite{SSS15}. In fact, by using the mean value formula, one may get
$$\fr{A_i(x)-A_i(y)}{|x-y|}=\int_{0}^1\Big\langle{\fr{x-y}{|x-y|}},{\nabla A_i(sx+(1-s)y)}\Big\rangle ds.$$
For each $i=1,\cdots, n$, plunge the above equality into $\C[\nabla A_1,\cdots,\nabla A_n, f](x)$ and write it as follows:
\Bes
\pv \int_{\R^d} K(x-y)\Big(\prod_{i=1}^n\Big[\sum_{j=1}^d\fr{x_j-y_j}{|x-y|}\int_0^1\pari_jA_i(sx+(1-s)y)ds\Big]\Big)\cdot f(y)dy.
\Ees

Then by the moment cancelation condition \eqref{e:12K_2}, the bound condition \eqref{e:12kb} and the regularity condition \eqref{e:12kr}, for any multi-indice $\alp\in\Z_+^d$ with  $|\alp|=n$, $K(x)({x}/{|x|})^{\alp}$ is a standard Calder\'on-Zygmund kernel. Therefore the proof reduces to show that the following operator
$$\C_{CJ}[a_1,\cdots,a_n,f](x)=\pv\int_{\R^d}k(x-y)(\prod_{i=1}^nm_{x,y}a_i)f(y)dy$$ maps $L^{q_1}(\R^d)\times\cdots\times L^{q_n}(\R^d)\times L^p(\R^d)$ to $L^r(\R^d)$, where $k$ is a standard Calder\'on-Zygmund kernel and $m_{x,y}a=\int_0^1a(sx+(1-s)y)dy$.
However, this estimate has been proved by A. Seeger, C. K. Smart and B. Street in \cite{SSS15}.
\end{proof}

\begin{prop}\label{p:12inftyweak}
Let $r=1$, $q_1=\cdots=q_n=\infty$, $p=1$. Then
\Bes
\|\C[\nabla A_1,\cdots,\nabla A_n, f]\|_{L^{1,\infty}(\R^d)}\lc\Big(\prod_{i=1}^n\|\nabla A_i\|_{L^{\infty}(\R^d)}\Big)\|f\|_{L^1(\R^d)}.
\Ees
\end{prop}
\begin{proof}
When $q_1=\cdots=q_n=\infty$, $A_i$ is a Lipschitz function for $i=1,\cdots,n$. Fix all $A_i$. We may regard $\C[\nabla A_1,\cdots,\nabla A_n, f](x)$ as a linear function of $f$. Then the kernel
$$K(x,y)=:K(x-y)\Big(\prod_{i=1}^n\fr{A_i(x)-A_i(y)}{|x-y|}\Big)$$
is a standard Calder\'on-Zygmund kernel (see \eg \cite[Page 211, Definition 4.1.2]{Gra250})which in fact satisfies the boundedness condition $|K(x,y)|\lc(\prod_{i=1}^n\|\nabla A_i\|_{L^\infty(\R^d)})|x-y|^{-d}$ and the following regularity conditions
$$
|K(x_1,y)-K(x_2,y)|\lc(\prod_{i=1}^n\|\nabla A_i\|_{L^\infty(\R^d)})\frac{|x_1-x_2|^{\del}}{|x_1-y|^{d+\del}}\ \ \text{for}\ |x_1-y|>2|x_1-x_2|,$$
$$
|K(x,y_1)-K(x,y_2)|\lc(\prod_{i=1}^n\|\nabla A_i\|_{L^\infty(\R^d)})\frac{|y_1-y_2|^{\del}}{|x-y_1|^{d+\del}}\ \ \text{for}\ |x-y_1|>2|y_1-y_2|.$$
Therefore by Proposition \ref{p:12strongr} with $q_1=\cdots=q_n=\infty, p=2$ and the standard Calder\'on-Zygmund theory (see \eg \cite[Page 226, Theorem 4.2.2]{Gra250}), we may get that the operator $\C[\nabla A_1,\cdots,\nabla A_n, \cdot]$ is of weak type (1,1) with bound $\prod_{i=1}^n\|\nabla A_i\|_{L^\infty(\R^d)}$, thus we complete the proof.
\end{proof}
\vskip0.24cm
\subsection{Case: $d/(d+n)\leq r\leq 1$ and $d\leq q_1,\cdots,q_n\leq\infty$}\label{s:1223}\quad
\vskip0.24cm
In this subsection, we consider the case $d/(d+n)\leq r\leq 1$ and $d\leq q_1,\cdots,q_n\leq\infty$. Without loss of generality, we may suppose the first $q_1,\cdots,q_l>d$ and $q_{l+1},\cdots,q_n=d$ with $0\leq l\leq n$. Here when $l=0$, we mean all $q_1=\cdots=q_n=d$. The proof of $p=\infty$ is slight different from that of $1\leq p<\infty$. So we will give two propositions in the following. Let us see the case $1\leq p<\infty$ firstly. We will point out in the proof where it doesn't work for $p=\infty$. And the proof of the case $p=\infty$ will be given later.
\begin{prop}\label{p:12qibigd}
Let $d/(d+n)\leq r\leq1$, $d<q_1,\cdots,q_l\leq\infty$ and $q_{l+1}$, $\cdots$, $q_n=d$ with $0\leq l\leq n$, $1\leq p<\infty$. Then
\Be\label{e:12weakqgeqd}
\begin{split}
\|\C[\nabla A_1,&\cdots,\nabla A_n, f]\|_{L^{r,\infty}(\R^d)}\\
&\lc\Big(\prod_{i=1}^l\|\nabla A_i\|_{L^{q_i}(\R^d)}\Big)\Big(\prod_{i=l+1}^{n}\|\nabla A_i\|_{L^{d,1}(\R^d)}\Big)\|f\|_{L^p(\R^d)},
\end{split}
\Ee
where $L^{d,1}(\R^d)$ is the standard Lorentz space.
\end{prop}

\begin{proof}
By using a standard limiting argument, we only need to show that when $A_i$ ($i=1,\cdots,n$) and $f$ are $C^\infty$ functions with compact supports, the following inequality
\Bes
\begin{split}
m(\{x\in\R^d:&|\C[\nabla A_1,\cdots,\nabla A_n, f](x)|>\lam\})\\
&\lc\lam^{-r}\Big(\prod_{i=1}^l\|\nabla A_i\|^r_{L^{q_i}(\R^d)}\Big)\Big(\prod_{i=l+1}^{n}\|\nabla A_i\|^r_{L^{d,1}(\R^d)}\Big)\|f\|^r_{L^p(\R^d)},
\end{split}
\Ees
holds for any $\lam>0$.
By a simple scaling argument, we may assume that
$$\|\nabla A_i\|_{L^{q_i}(\R^d)}=\|\nabla A_j\|_{L^{d,1}(\R^d)}=\|f\|_{L^p(\R^d)}=1,$$
for $i=1,\cdots,l$ and $j=l+1,\cdots,n$.
Fix $\lam>0$. For convenience we set
\Be\label{e:12elam}
E_\lam=\{x\in\R^d:|\mathcal{C}[\nabla A_1,\cdots,\nabla A_n, f](x)|>\lam\}.
\Ee

We need to show $|E_\lam|\lc\lam^{-r}.$
First suppose that all $q_1,\cdots, q_l<\infty$. Once we have understood the proof in this situation, we can modify the proof to the other case that there exist some $q_i=\infty$ for $i=1,\cdots,l$. We shall show how to do this in the last part of the proof.
Define the {\it exceptional set}
\Bes
\begin{split}
J_{i,\lam}=\big\{x\in\R^d:\M (\nabla A_i)(x)>\lam^{\fr{r}{q_i}}\big\}.
\end{split}
\Ees
for $i=1,\cdots,n$. Here it should be pointed out that if $q_i=\infty$, the above definition is meaningless. Therefore we need to assume all $q_i<\infty$ firstly.
From Lemma \ref{l:mw} and Lemma \ref{l:11md}, $\M$ maps $L^p(\R^d)$ to itself for $p>d$ and maps $L^{d,1}(\R^d)$ to $L^{d,\infty}(\R^d)$, i.e.
\Be\label{e:11jlam}
\begin{split}
|J_{i,\lam}|\lc\lam^{-r}\|\nabla A_i\|^{q_i}_{L^{q_i}(\R^d)}&=\lam^{-r}, \ \ i=1,\cdots,l;\\
|J_{j,\lam}|\lc\lam^{-r}\|\nabla A_j\|^{d}_{L^{d,1}(\R^d)}&=\lam^{-r}, \ \ j=l+1,\cdots,n.
\end{split}
\Ee

Set $J_\lam=\cup_{i=1}^n J_{i,\lam}$. Choose an open set $G_\lam$ which satisfies the following conditions:
(1) $J_\lam\subset G_\lam$;
(2) $m(G_\lam)\leq2|J_\lam|$.
By the property \eqref{e:11jlam} of $J_{i,\lam}$, we see that $m(G_\lam)\lc\lam^{-r}$. Next making a Whitney decomposition of $G_\lam$ (see \eg \cite{Gra249}), one may get a family of disjoint dyadic cubes $\{Q_k\}_k$ such that
\begin{enumerate}[(i).]
\item \quad $G_\lam=\bigcup_{k=1}^\infty Q_k$;
\item \quad $\sqrt{d}\cdot l(Q_k)\leq dist(Q_k,(G_\lam)^c)\leq4\sqrt{d}\cdot l(Q_k).$
\end{enumerate}
With those properties (i) and (ii), for each $Q_k$, we could construct a larger cube $Q_k^*$ so that $Q_k\subset Q_k^*$, $Q_k^*$ is centered at $y_k$ and $y_k\in (G_\lam)^c$, $|Q_k^*|\approx|Q_k|$. By the property (ii) above, the distance between $Q_k$ and $(G_\lam)^c$ equals to $Cl(Q_k)$.
Therefore by the construction of $Q_k^*$ and $y_k$, one may get
\Be\label{e:12whitney}
dist(y_k,Q_k)\approx l(Q_k).
\Ee

Now we return to give an estimate of $E_\lam$. Split $f$ into two parts $f=f_1+f_2$ where $f_1(x)=f(x)\chi_{(G_\lam)^c}(x)$ and $f_2(x)=f(x)\chi_{G_\lam}(x)$. By the definition of $J_\lam$, when restricted on $(G_\lam)^c$, $A_i$ is a Lipschitz function with $\|\nabla A_i\|_{L^\infty((G_\lam)^c)}\leq\lam^{\fr{r}{q_i}}$ for $i=1,\cdots,n$. Let $\tilde{A}_i$ stand for the Lipschitz extension of $A_i$ from $(G_\lam)^c$ to $\R^d$ (see \cite[page 174, Theorem 3]{Ste70}) so that
for each $i=1,\cdots,n$,
$$\tilde{A}_i(y)=A_i(y)\ \ \text{if} \ y\in (G_\lam)^c;$$
$$\big|\tilde{A}_i(x)-\tilde{A}_i(y)\big|\leq\lam^{\fr{r}{q_i}}|x-y|\ \ \text{for all}\ x,y\in\R^d.$$

Since the operator $\mathcal{C}[\cdots,\cdot]$ is multilinear, we split $E_\lam$ as three terms and give estimates as follows:
\Be\label{e:12spmu}
\begin{split}
m(\{x&\in\R^d: |\mathcal{C}[\nabla A_1,\cdots,\nabla A_n,f](x)|>\lambda\})\\
&\leq m(10G_\lam)+m\big(\{x\in (10G_\lam)^c:|\mathcal{C}[\nabla A_1,\cdots,\nabla A_n,f_1](x)|>\lambda/2\}\big)\\
&\ \ \ \ +m\big(\{x\in (10G_\lam)^c:|\mathcal{C}[\nabla A_1,\cdots,\nabla A_n,f_2](x)|>\lambda/2\}\big).
\end{split}
\Ee

The first term above satisfies $|10G_\lam|\lc\lam^{-r}$, which is the required bound. In the following, we only consider $x\in(10G_\lam)^c$. By the definition of $f_1$, one may see that
$$\mathcal{C}[\nabla A_1,\cdots,\nabla A_n,f_1](x)=\mathcal{C}[\nabla\tilde{A}_1,\cdots,\nabla\tilde{A}_n,f_1](x).$$
With this equality in hand, Proposition \ref{p:12strongr} ($1<p<\infty$) and Proposition \ref{p:12inftyweak} ($p=1$) imply
\Be\label{e:12qbigd}
\begin{split}
m\big(\big\{x&\in (10G_\lam)^c: |\mathcal{C}[\nabla A_1,\cdots,\nabla A_n,f_1](x)|>{\lam}/{2}\big\}\big)\\
&=m\big(\big\{x\in (10G_\lam)^c: |\mathcal{C}[\nabla \tilde{A}_1,\cdots,\nabla\tilde{A}_n,f_1](x)|>{\lam}/{2}\big\}\big)\\
&\ \ \lc \lam^{-p}\Big(\prod_{i=1}^n\|\nabla\tilde{A}_i\|^p_{L^\infty(\R^d)}\Big)\|f_1\|^p_{L^p(\R^d)}
\lc\lam^{-p+p\sum_{i=1}^n\fr{r}{q_i}}=\lam^{-r}.
\end{split}
\Ee
If $p=\infty$, the above method does not work.
We will show how to prove this kind of estimate in the next proposition.

Let us turn to $\mathcal{C}[\nabla A_1,\cdots,\nabla A_n, f_2](x)$. Define $\N_i^j=\{i,i+1,\cdots,j\}$. Recall our construction of $G_\lam$, $y_k$, $Q_k$ and $Q_k^*$ in the paragraph above \eqref{e:12whitney}. Then we can write $f_2=\sum_k f\chi_{Q_k}$. Therefore we may get
$$\mathcal{C}[\nabla A_1,\cdots,\nabla A_n, f_2](x)=\sum_k\mathcal{C}[\nabla A_1,\cdots,\nabla A_n, f\chi_{Q_k}](x).$$
Below we should carefully study $\prod_{i=1}^n\fr{{A}_i(x)-A_i(y)}{|x-y|}$. We will separate it into several terms and then give an estimate for each term.  Write
\Bes
\begin{split}
&\ \ \ \ \prod_{i=1}^n\fr{{A}_i(x)-A_i(y)}{|x-y|}\\&=\prod_{i=1}^n\Big(\fr{\tilde{A}_i(x)-\tilde{A}_i(y)}{|x-y|}+\fr{\tilde{A}_i(y)-\tilde{A}_i(y_k)}{|x-y|}+\fr{A_i(y_k)-A_i(y)}{|x-y|}\Big)\\
&=\sum\Big(\prod_{i\in N_1}\fr{\tilde{A}_i(x)-\tilde{A}_i(y)}{|x-y|}\Big)\Big(\prod_{i\in N_2}\fr{\tilde{A}_i(y)-\tilde{A}_i(y_k)}{|x-y|}\Big)\Big(\prod_{i\in N_3}\fr{{A}_i(y_k)-{A}_i(y)}{|x-y|}\Big)\\
&=I(x,y)+II(x,y,y_k),
\end{split}
\Ees
where in the third equality we divide $\N_{1}^n=N_1\cup N_2\cup N_3$ with $N_1$, $N_2$, $N_3$ non intersecting each other; and $I(x,y)$, $II(x,y,y_k)$, are defined as follows
\Be\label{e:12axyqbigd}
\begin{split}
I(x,y)=&\prod_{i=1}^n\fr{\tilde{A}_i(x)-\tilde{A}_i(y)}{|x-y|},\\
II(x,y,y_k)=&\sum_{N_1\subsetneq\N_{1}^n}\Big(\prod_{i\in N_1}\fr{\tilde{A}_i(x)-\tilde{A}_i(y)}{|x-y|}\Big)\times\\
&\times
\Big(\prod_{i\in N_2}\fr{\tilde{A}_i(y)-\tilde{A}_i(y_k)}{|x-y|}\Big)\Big(\prod_{i\in N_3}\fr{{A}_i(y_k)-{A}_i(y)}{|x-y|}\Big).
\end{split}
\Ee
By the above decomposition, we in fact divide $\mathcal{C}[\nabla A_1,\cdots,\nabla A_n, f\chi_{Q_k}](x)$ into $3^n$ terms. We separate these terms into two parts according $I$ and $II$.

\emph{Estimate of $\mathcal{C}[\cdots,\cdot]$ related to $I$.}  This estimate is similar to \eqref{e:12qbigd}. In fact, in this case there is only one term $\C[\nabla\tilde{A}_1,\cdots,\nabla\tilde{A}_n,f_2]$. Then by Proposition \ref{p:12strongr} ($1<p<\infty$) and Proposition \ref{p:12inftyweak} ($p=1$), we get
\Bes
\begin{split}
m\big(\big\{x\in& (10G_\lam)^c: |\mathcal{C}[\nabla \tilde{A}_1,\cdots,\nabla\tilde{A}_n,f_2](x)|>{\lam}/{2}\big\}\big)\\
&\lc \lam^{-p}\Big(\prod_{i=1}^n\|\nabla\tilde{A}_i\|^p_{L^\infty(\R^d)}\Big)\|f_2\|^p_{L^p(\R^d)}
\lc\lam^{-p+p\sum_{i=1}^n\fr{r}{q_i}}=\lam^{-r}.
\end{split}
\Ees
If $p=\infty$, the above argument may not work again.

\emph{Estimate of $\mathcal{C}[\cdots,\cdot]$ related to $II$.}
It suffices to consider one term $\mathcal{C}[\cdots,\cdot]$ related to $II$ in which $N_1$ is a proper subset of $\N_1^n$.
In this case, without loss of generality, we may assume $N_1=\{1,\cdots,v\}$, $N_2=\{v+1,\cdots,m\}$ and $N_3=\{m+1,\cdots, n\}$ with $0\leq v\leq m\leq n$ and $v<n$. Here when $v=0$, it means that $N_1=\emptyset$; when $v=m$, $N_2=\emptyset$; when $m=n$,  $N_3=\emptyset$. With these notation, one can easily see that $N_1$ is a proper subset of $\N_1^n$.
By a slight abuse of notation, we still use $II(x,y,y_k)$ to represent one term related to $N_1$, $N_2$ and $N_3$ in \eqref{e:12axyqbigd} and use $H_{II}(x)$ to represent $\mathcal{C}[\cdots,\cdot]$ related to ${II}(x,y,y_k)$, i.e.
$$H_{II}(x)=\sum_k\int_{Q_k}K(x-y)II(x,y,y_k)f(y)dy.$$
Notice that $y_k$ lies in the $(G_\lam)^c$, thus $y_k\in (J_{i,\lam})^c$. Therefore we get
\Be
\M(\nabla A_i)(y_k)\leq\lam^{\fr{r}{q_i}}, \ \text{for $i=m+1,\cdots,n$.}
\Ee
With the above fact and $\tilde{A}_i$ is a Lipschitz function with bound $\lam^{r/q_i}$ for $i=1,\cdots,m$, we get
\Bes
\begin{split}
|II(x,y,y_k)|&\lc\lam^{\sum_{i=1}^m\fr{r}{q_i}}\fr{|y-y_k|^{n-v}}{|x-y|^{n-v}}\prod_{i=m+1}^n\M(\nabla A_i)(y_k)\\
&\lc\lam^{\sum_{i=1}^n\fr{r}{q_i}}\fr{|y-y_k|^{n-v}}{|x-y|^{n-v}}.
\end{split}
\Ees
Notice that we only consider $x\in(10G_\lam)^c$, then for $y\in Q_k$,  $|x-y|\geq 2l(Q_k)\approx|y-y_k|$ by \eqref{e:12whitney}. Combining the above discussion with \eqref{e:12kb}, we get
\Be\label{e:12qgeqdinfty}
\begin{split}
|H_{II}(x)|&\leq\sum_k\int_{Q_k}|K(x-y)|\cdot|II(x,y,y_k)|\cdot|f(y)|dy\\
&\lc\lam^{\sum_{i=1}^n\fr{r}{q_i}}\sum_k\int_{Q_k}\fr{l(Q_k)^{n-v}}{[l(Q_k)+|x-y|]^{d+n-v}}|f(y)|dy.
\end{split}
\Ee

Applying the Chebyshev inequality with the above estimate, and utilizing Lemma \ref{l:12disq1infty} with $|\Om|\equiv1$ (note that $n-v\geq1$), we finally get
\Bes
\begin{split}
m(\{x\in(10G_\lam)^c: |H_{II}(x)|>\lam\})&\leq\lam^{-p+\sum_{i=1}^n\fr{rp}{q_i}}\int_{(10G_\lam)^c}|T_{n-v}f(x)|^pdx\\
&\lc\lam^{-r}\|f\|^p_{L^p(\R^d)}.
\end{split}
\Ees
Hence we complete the proof of the term $II$. If $p=\infty$, the last argument above may not work and a little different discussion should be involved, see the next proposition.

Finally, we add some word about how to modify the above proof to the case $q_i=\infty$ for some $i=1,\cdots,l$. We may suppose only $q_1=\cdots=q_u=\infty$ with $1\leq u\leq l$. Thus $A_1$, $\cdots$, $A_u$ are Lipschitz functions which in fact are nice functions. Then we just fix $A_1, \cdots, A_u$ in the rest of the proof. We only make a construction of {\it exceptional set\/} for $A_{u+1}, \cdots, A_n$ and study $\prod_{i={u+1}}^n\fr{A_i(x)-A_i(y)}{|x-y|}$ by using the same way as we have done previously.
After that utilizing $A_1$, $\cdots$, $A_u$ are Lipschitz functions to deal with all estimates involved with $A_1,\cdots,A_u$,
we may get the required bound.
\end{proof}

\begin{prop}\label{p:12qleqinfty}
Let $d/(d+n)\leq r\leq1$, $d<q_1,\cdots,q_l\leq\infty$ and $q_{l+1}$, $\cdots$, $q_n=d$ with $0\leq l\leq n$, $p=\infty$. Then the weak type estimate \eqref{e:12weakqgeqd} holds.
\end{prop}
\begin{proof}
The proof is quite similar to that of Proposition \ref{p:12qibigd}. So we shall be brief and  only indicate necessary modifications here. Proceeding the proof as we do that in Proposition \ref{p:12qibigd}, there are four different arguments involved. We will point out below one by one.

The first one is that when we choose the set $E_\lam$, we choose
$$E_\lam=\{x\in\R^d:|\mathcal{C}[\nabla A_1,\cdots,\nabla A_n, f](x)|>C_0\lam\},$$
where $C_0$ is a constant which will be determined later. Our goal is to show $m(E_\lam)\lc\lam^{-r}$. We split $E_\lam$ as several terms and give estimates as follows:
\Bes
\begin{split}
m(\{x&\in\R^d: |\mathcal{C}[\nabla A_1,\cdots,\nabla A_n,f](x)|>C_0\lambda\})\\
&\leq m(10G_\lam)+m\big(\{x\in (10G_\lam)^c:|\mathcal{C}[\nabla A_1,\cdots,\nabla A_n,f_1](x)|>C_0\lambda/2\}\big)\\
&\qquad \ \ \ +m\big(\{x\in (10G_\lam)^c:|\mathcal{C}[\nabla A_1,\cdots,\nabla A_n,f_2](x)|>C_0\lambda/2\}\big).
\end{split}
\Ees
The first term above satisfies $|10G_\lam|\lc\lam^{-r}$, so it suffices to consider the second and third term. We only consider $x\in(10G_\lam)^c$.

The second difference is the estimate related to $f_1$. Here we need to choose $\tilde{r}$, $\tilde{q}_1$, $\cdots$, $\tilde{q}_n$, such that $1<\tilde{r}<\infty$, $q_1<\tilde{q}_1<\infty$, $\cdots$, $q_n<\tilde{q}_n<\infty$ and $\fr{1}{\tilde{r}}=\sum_{i=1}^n\fr{1}{\tilde{q}_i}$. Apply Lemma \ref{p:12strongr} with those above $\tilde{r}$, $\tilde{q}_1$, $\cdots$, $\tilde{q}_n$,
\Bes
\begin{split}
m&\big(\big\{x\in (10G_\lam)^c: |\mathcal{C}[\nabla A_1,\cdots,\nabla A_n,f_1](x)|>{C_0\lam}/{2}\big\}\big)\\
&\leq m\big(\big\{x\in (G_\lam)^c: |\mathcal{C}[\nabla ({A}_1\chi_{(G_\lam)^c}),\cdots,\nabla({A}_n\chi_{(G_\lam)^c}),f_1](x)|>{C_0\lam}/{2}\big\}\big)\\
&\lc \lam^{-\tilde{r}}\Big(\prod_{i=1}^n\|\nabla({A}_i\chi_{(G_\lam)^c})\|^{\tilde{r}}_{L^{\tilde{q}_i}(\R^d)}\Big)\|f_1\|^{\tilde{r}}_{L^\infty(\R^d)}\\
&\lc
\lam^{-\tilde{r}}\Big(\prod_{i=1}^n\|\nabla{A}_i\|^{(\tilde{q}_i-q_i)\fr{\tilde{r}}{\tilde{q}_i}}_{L^{\infty}({(G_\lam)^c})}\Big)
\Big(\prod_{i=1}^n \|\nabla A_i\|_{L^{q_i}(\R^d)}^{\fr{q_i}{\tilde{q}_i}\tilde{r}}\Big)\|f_1\|^{\tilde{r}}_{L^\infty(\R^d)}\\
&\lc\lam^{-\tilde{r}+\tilde{r}\big(\sum_{i=1}^n\fr{r}{q_i}\big)-{r}\big(\sum_{i=1}^n\fr{\tilde{r}}{\tilde{q}_i}\big)}=\lam^{-r},
\end{split}
\Ees
where in the last second inequality we use $A_i$ is a Lipschitz function on $(G_\lam)^c$ with Lipschitz bound $\lam^{\fr{r}{q_i}}$ for $i=1,\cdots,n$ and $L^{d,1}(\R^d)\subsetneq L^d(\R^d)$ if $q_{i}=d$.

Next consider the estimate related to $f_2$. As we have done in the proof of Proposition \ref{p:12qibigd}, we divide $\mathcal{C}[\nabla A_1,\cdots,\nabla A_n, f_2](x)$ into several terms and separate these terms into two part according $I$ and $II$ in \eqref{e:12axyqbigd}. Then we get
\Bes
\begin{split}
&m\big(\{x\in (10G_\lam)^c:|\mathcal{C}[\nabla A_1,\cdots,\nabla A_n,f_2](x)|>C_0\lambda/2\}\big)\\
&\qquad\leq m\big(\big\{x\in (10G_\lam)^c: |\mathcal{C}[\nabla \tilde{A}_1,\cdots,\nabla\tilde{A}_n,f_2](x)|>{C_0\lam}/{4}\big\}\big)\\
&\qquad\qquad+ m\big(\big\{x\in (10G_\lam)^c: |H_{II}(x)|>{C_0\lam}/{4}\big\}\big).
\end{split}
\Ees

The third difference is the {\it estimate of $\mathcal{C}[\cdots,\cdot]$ related to $I$\/}. Here we apply Lemma \ref{p:12inftyweak} and the estimate $\|f_2\|_{L^1(\R^d)}\lc\|f\|_{L^\infty(\R^d)}|G_\lam|\lc\lam^{-r}$ to get
\Bes
\begin{split}
m\big(\big\{x\in& (10G_\lam)^c: |\mathcal{C}[\nabla \tilde{A}_1,\cdots,\nabla\tilde{A}_n,f_2](x)|>{C_0\lam}/{4}\big\}\big)\\
&\lc \lam^{-1}\Big(\prod_{i=1}^n\|\nabla\tilde{A}_i\|_{L^\infty(\R^d)}\Big)\|f_2\|_{L^1(\R^d)}
\lc\lam^{-1+\big(\sum_{i=1}^n\fr{r}{q_i}\big)-r}=\lam^{-r}.
\end{split}
\Ees

The fourth difference is the {\it estimate of $\mathcal{C}[\cdots,\cdot]$ related to $II$\/}. We will show that
\Be\label{e:12Hiiemp}\big\{x\in (10G_\lam)^c: |H_{II}(x)|>{C_0\lam}/{4}\big\}=\emptyset.
\Ee
In fact, by \eqref{e:12qgeqdinfty} and Lemma \ref{l:12disq1infty} with $q=\infty$, we get for any $x\in(10G_\lam)^c$,
\Bes
\begin{split}
|H_{II}(x)|\leq C_{d}\lam^{\sum_{i=1}^n\fr{r}{q_i}}\|f\|_{L^\infty(\R^d)}=C_d\lam.
\end{split}
\Ees
So if we choose $C_0>4C_d$, we get \eqref{e:12Hiiemp}. Thus we finish the proof.
\end{proof}
\vskip0.24cm
\subsection{Case: $d/(d+n)\leq r\leq 1$ and $1\leq q_1,\cdots,q_n<d$}\label{s:1224}\quad
\vskip0.24cm
In this subsection, we consider the case $d/(d+n)\leq r\leq 1$ and $1\leq q_1,\cdots,q_n<d$. Again here the proof of $p=\infty$ is a little different from that of $1\leq p<\infty$. We first consider $1\leq p<\infty$ and point out in the proof where it doesn't work for $p=\infty$.
\begin{prop}\label{p:12qleqdall}
Let $ d/(d+n)\leq r\leq1$, $1\leq q_1,\cdots,q_n< d$, $1\leq p<\infty$. Then the weak type estimate \eqref{e:12muweak} holds.
\end{prop}
\begin{proof}
Our main goal is to prove that for any $\lam>0$, the following inequality holds
\Bes
\begin{split}
m(\{x\in\R^d:|\C[\nabla A_1,\cdots,&\nabla A_n, f](x)|>\lam\})\\
&\lc\lam^{-r}\Big(\prod_{i=1}^n\|\nabla A_i\|^r_{L^{q_i}(\R^d)}\Big)\|f\|^r_{L^p(\R^d)}.
\end{split}
\Ees

Now we fix $\lam>0$. Recall $E_\lam$ defined in \eqref{e:12elam}.
By rescaling as showed in the proof of Proposition \ref{p:12qibigd}, we only need to show $|E_\lam|\lc\lam^{-r}$. The main idea is to construct some {\it exceptional set\/} such that the measure of {\it exceptional set\/} is bounded by $\lam^{-r}$, which is our required estimate. At the same time on the complementary set of {\it exceptional set\/}, these functions $A_i$ should be Lipschitz functions with bound $\lam^{\fr{r}{q_i}}$ for each $i=1,\cdots,n$. Below we begin our constructions of some {\it exceptional set\/} which will be involved with several steps.

\subsubsection*{Step 1: Calder\'on-Zygmund decomposition.}
By the formula given in \cite[page 125, (17)]{Ste70}, for each $A_i$, $i=1,\cdots,n$, one may write
\Bes
A_i(x)=\sum_{j=1}^dC_d\int_{\R^d}\fr{x_j-y_j}{|x-y|^d}\pari_j A_i(y)dy=:\sum_{j=1}^dA_{i,j}(x).
\Ees
For each $|\pari_jA_i|^{q_i}\in L^1(\R^d)$ with $j=1,\cdots, d$ and $i=1,\cdots,n$, making a Calder\'on-Zygmund decomposition at level $\lam^r$,
one may have the following conclusions (see \eg \cite{Gra249}):
\begin{enumerate}[\quad (cz-i)]
\rm\item $\pari_j A_i=g_{j,i}+b_{j,i}$, $\|g_{j,i}\|_{L^\infty(\R^d)}\lc\lambda^{\fr{r}{q_i}}$, $\|g_{j,i}\|_{L^{q_i}(\R^d)}\lc\|\nabla A_i\|_{L^{q_i}(\R^d)}$;
\item $b_{j,i}=\sum_{Q\in \mathcal{Q}_{j,i}} b_{j,i,Q}$, $\supp\  b_{j,i,Q}\subset Q$, where $\mathcal{Q}_{j,i}$ is a countable set of disjoint dyadic cubes;
\item Let $E_{j,i}=\bigcup_{Q\in \mathcal{Q}_{j,i}} Q$, then $m(E_{j,i})\lc {\lambda^{-{r}}}\|\pari_jA_i\|^{q_i}_{L^{q_i}(\R^d)}$;
\item $\int b_{j,i,Q}(y)dy=0$ for each $Q\in \mathcal{Q}_{j,i}$ and $\|b_{j,i,Q}\|^{q_i}_{L^{q_i}(\R^d)}\lc{\lambda^{{r}}}|Q|$, so we get $\|b_{j,i}\|_{L^{q_i}{(\R^d)}}\lc\|\pari_jA_i\|_{L^{q_i}(\R^d)}$ by  (cz-ii) and (cz-iii).
\end{enumerate}

We are going to separate $A_{i,j}$ into two parts according the above Calder\'on-Zygmund decomposition property (cz-i):
\Bes
\begin{split}
A^g_{i,j}(x)&=C_d\int_{\R^d}\fr{x_j-y_j}{|x-y|^d}g_{j,i}(y)dy;\\
A^b_{i,j}(x)&=C_d\int_{\R^d}\fr{x_j-y_j}{|x-y|^d}b_{j,i}(y)dy.
\end{split}
\Ees
Set the \emph{exceptional set}$B_\lam=\cup_{i=1}^n\cup_{j=1}^dE_{j,i}$. Then by (cz-iii), we get
$m(B_\lam)\lc\lam^{-r}.$

\subsubsection*{Step 2: Exceptional set $D_\lam$.} Set $\fr{1}{s_i}=\fr{1}{q_i}-\fr{1}{d}$ for $i=1,\cdots,n$. Define the following \emph{exceptional set}
\Bes
D_{i,\lam}=\Big\{x\in\R^d:\mathfrak{M}_{s_i}(\nabla A_i)(x)>\lam^{\fr{r}{q_i}}\Big\}
\Ees
where the maximal operator $\mathfrak{M}_{s_i}$ is defined in Lemma \ref{l:12qleqd}. We denote $D_\lam=\cup_{i=1}^nD_{i,\lam}$. Then by Lemma \ref{l:12qleqd} and the weak type (1,1) bound for the Hardy-Littlewood maximal operator, we see
\Bes
m(D_{i,\lam})\leq m(\{x\in\R^d:M_{q_i}(\nabla A_i)(x)>C\lam^{\fr{r}{q_i}}\})\lc\lam^{-r}\|\nabla A_i\|^{q_i}_{L^{q_i}(\R^d)}=\lam^{-r}.
\Ees
So does $m(D_\lam)\lc\lam^{-r}$.

\subsubsection*{Step 3: Exceptional set $F_\lam$.} For each $j=1,\cdots,d$, $i=1,\cdots,n$, we define the functions
\Bes
\Delta^{j,i}(x)=\sum_{Q\in\mathcal{Q}_{j,i}}\fr{l(Q)}{[l(Q)+|x-y_Q|]^{d+1}}m(Q)
\Ees
where $y_Q$ is the center of $Q$. We define another \emph{exceptional set}
\Bes
F_{j,i,\lam}=\{x\in\R^d:\Delta^{j,i}(x)>1\},\ \ F_{\lam}=\cup_{j=1}^d\cup_{i=1}^n F_{j,i,\lam}.
\Ees
Then by the Chebyshev inequality and (cz-iii), we get
\Bes
|F_{j,i,\lam}|\leq\int_{\R^d}\Delta_{j,i,\lam}(x)dx\leq\Big[\int_{\R^d}\fr{1}{(1+|y|)^{d+1}}dy\Big]
\Big[\sum_{Q\in\mathcal{Q}_{j,i}}|Q|\Big]\lc\lam^{-r}.
\Ees
So does $m(F_\lam)\lc\lam^{-r}$.

\subsubsection*{Step 4: Exceptional set $H_\lam$.} We define the \emph{exceptional set}
\Bes
H_{i,j,\lam}=\{x\in\R^d:\M(\nabla A^g_{i,j})(x)>\lam^{r/q_i}\},\ \ H_\lam=\cup_{i=1}^n\cup_{j=1}^d H_{i,j,\lam}.
\Ees
Notice that by the definition of $A^g_{i,j}$, for each $z=1,\cdots,d$, we have $$\F(\pari_{z}A_{i,j}^g)(\xi)=C\fr{\xi_z\xi_j}{|\xi|^2}\F(g_{j,i})(\xi)\Rightarrow\nabla A^g_{i,j}=CRR_jg_{j,i},$$
where $\F$ is the Fourier transform, $R_j$ is the Riesz transform and $R=(R_1,\cdots,R_d)$. Since $R_j$ is of strong $(q,q)$ type for $1<q<\infty$, we get $\|\nabla A^g_{i,j}\|_{L^q(\R^d)}\lc\|g_{j,i}\|_{L^q(\R^d)}.$ By the Chebyshev inequality, Lemma \ref{l:mw} and (cz-i), we get for $d<q<\infty$,
\Bes
\begin{split}
m(H_{i,j,\lam})&\lc\lam^{-\fr{qr}{q_i}}\int_{\R^d}[\M(\nabla A_{i,j}^g)(x)]^qdx\lc\lam^{-\fr{qr}{q_i}}\int_{\R^d}|\nabla A_{i,j}^g(x)|^qdx\\
&\lc\lam^{-\fr{qr}{q_i}}\int_{\R^d}|g_{j,i}(x)|^qdx\lc\lam^{-r}\int_{\R^d}|g_{j,i}(x)|^{q_i}dx\lc\lam^{-r}.
\end{split}
\Ees
So does $m(H_\lam)\lc\lam^{-r}$.

\subsubsection*{Step 5: Final exceptional set $G_\lam$} Based on the construction of $B_\lam, D_\lam, F_\lam, H_\lam$ in Step 1-4,
we choose an open set $G_\lam$ which satisfies the following conditions:
\begin{enumerate}[(1).]
\item\quad $\big(10B_\lam\cup 10D_\lam\cup 10F_\lam\cup 10H_\lam\big)\subset G_\lam$;
\item\quad $m(G_\lam)\leq(20)^d(|B_\lam|+|D_\lam|+|F_\lam|+|H_\lam|)$.
\end{enumerate}
By the property of $B_{\lam}$, $D_\lam$, $F_\lam$ and $H_\lam$, we see that $m(G_\lam)\lc\lam^{-r}$. Next making a Whitney decomposition of $G_\lam$ (see \cite{Gra249}), we may get a family of disjoint dyadic cubes $\{Q_k\}_k$ such that
\begin{enumerate}[(i).]
\item \quad $G_\lam=\bigcup_{k=1}^\infty Q_k$;
\item \quad $\sqrt{d}\cdot l(Q_k)\leq dist(Q_k,(G_\lam)^c)\leq4\sqrt{d}\cdot l(Q_k).$
\end{enumerate}
With those properties (i) and (ii), for each $Q_k$, we could construct a larger cube $Q_k^*$ so that $Q_k\subset Q_k^*$, $Q_k^*$ is centered at $y_k$ and $y_k\in (G_\lam)^c$, $|Q_k^*|\approx|Q_k|$. By the property (ii) above, the distance between $Q_k$ and $(G_\lam)^c$ equals to $Cl(Q_k)$.
Therefore by the construction of $Q_k^*$ and $y_k$, we get
\Be\label{e:12whitney2}
dist(y_k,Q_k)\approx l(Q_k).
\Ee
\vskip 0.24cm
Clearly, the \emph{exceptional set} $G_\lam$ constructed in Step 5 satisfies that the measure is bounded by $\lam^{-r}$. In the following we will show that these functions $A_i$ are Lipschitz functions on the complementary set of $G_\lam$.

\subsubsection*{Step 6: Lipschitz estimates of $A_i$ on $(G_\lam)^c$} By the Calder\'on-Zygmund decomposition in Step 1, it suffices to show that $A_{i,j}^g$ and $A_{i,j}^b$ satisfy Lipschitz estimates on $\big(G_\lam\big)^c$ for each $i=1,\cdots,n$ and $j=1,\cdots,d$. Firstly, it is easy to see that $A_{i,j}^g$ satisfies Lipschitz estimates by the construction of $H_\lam$ in Step 4. In fact, for any $x,y\in H^c_\lam$, we get
\Be\label{e:12lipg}
|A_{i,j}^g(x)-A_{i,j}^g(y)|\leq\lam^{\fr{r}{q_i}}|x-y|.
\Ee

We give effort to showing $A_{i,j}^b$ is a Lipschitz function on $(G_\lam)^c$. Recall the Calder\'on-Zygmund decomposition property (cz-ii), (cz-iii) and (cz-iv) in Step 1. For each $b_{j,i}=\sum_{Q\in \mathcal{Q}_{j,i}} b_{j,i,Q}$, $\supp\ b_{j,i,Q}\subset Q$, where $\mathcal{Q}_{j,i}$ is a countable set of disjoint dyadic cubes. Then for each $Q\in\mathcal{Q}_{j,i}$, we define
$$A_{i,j}^{b_Q}(x)=C_d\int_{\R^d}\fr{x_j-z_j}{|x-z|^d}b_{j,i,Q}(z)dz.$$
Now we choose $x,y\in(G_\lam)^c$ and fix a dyadic cube $Q\in\mathcal{Q}_{j,i}$. Then by the construction of $G_\lam$, $x,y\in(10B_\lam)^c$, i.e. $x,y\in (10Q)^c$, therefore we get
$dist(x,Q)\geq \fr{9}{2}l(Q)$ and $dist(y,Q)\geq \fr{9}{2}l(Q)$.
We will give a straight-forward Lipschitz estimate of $A_{i,j}^{b_Q}$. Let $z_Q$ be the center of $Q$. Without loss of generality, suppose that $|x-z_Q|\leq|y-z_Q|$. Choose a point $Z\in\R^d$ such that
\Bes
\begin{split}
|x-Z|<100|x-y|;\ \ |y-Z|\leq100|x-y|;\ \ |X-z_Q|>\fr{2}{5}|x-z_Q|
\end{split}
\Ees
for any $X$ belongs to the polygonal with vertex $x,y,Z$.
One may draw a figure to check that such a point $Z$ always exists, provided that $dist(x,Q)>\fr{9}{2}l(Q)$ and $dist(y,Q)>\fr{9}{2}l(Q)$.
Now we split $A_{i,j}^{b_Q}(x)-A_{i,j}^{b_Q}(y)=A_{i,j}^{b_Q}(x)-A_{i,j}^{b_Q}(Z)+A_{i,j}^{b_Q}(Z)-A_{i,j}^{b_Q}(y).$
By using the mean value formula, we see that
\Be\label{e:12Aijc}
\begin{split}
&A_{i,j}^{b_Q}(x)-A_{i,j}^{b_Q}(Z)=\int_0^1\inn{x-Z}{\nabla (A_{i,j}^{b_Q})(tx+(1-t)Z)}dt;\\
&A_{i,j}^{b_Q}(Z)-A_{i,j}^{b_Q}(y)=\int_0^1\inn{Z-y}{\nabla (A_{i,j}^{b_Q})(tZ+(1-t)y)}dt.
\end{split}
\Ee
For any $0\leq t\leq1$, $tx+(1-t)Z$ and $tZ+(1-t)y$ lie in the polygonal with vertex $x,y,Z$. Notice that $|x-z_Q|\geq5l(Q)$. By our choice of $Z$, we have
\Be\label{e:12xyZt}
\begin{split}
&|tx+(1-t)Z-z_Q|>\fr{2}{5}|x-z_Q|>2l(Q),\\
&|tZ+(1-t)y-z_Q|>\fr{2}{5}|x-z_Q|>2l(Q).
\end{split}
\Ee
We set $Z(t)$ equals to $tx+(1-t)Z$ or $tZ+(1-t)y$ and $K_j(x)=x_j/|x|^d$. Using the cancelation condition of $b_{j,i,Q}$, \eqref{e:12xyZt} and (cz-iv) in Step 1, we see that
\Bes
\begin{split}
|\nabla &(A_{i,j}^{b_Q})(Z(t))|=\Big|\int_{\R^d}\Big[(\nabla K_j)(Z(t)-z)-(\nabla K_j)(Z(t)-z_Q)\Big]b_{j,i,Q}(z)dz\Big|\\
&\lc\fr{l(Q)}{[l(Q)+|x-z_Q|]^{d+1}}\|b_{j,i,Q}\|_{L^1(\R^d)}\lc\lam^{\fr{r}{q_i}}\fr{l(Q)}{[l(Q)+|x-z_Q|]^{d+1}}|Q|.
\end{split}
\Ees
Combining the above arguments with \eqref{e:12Aijc} and the construction of $Z$, we get
\Bes
\big|A_{i,j}^{b_Q}(x)-A_{i,j}^{b_Q}(y)\big|\lc\lam^{\fr{r}{q_i}}\fr{l(Q)}{[l(Q)+|x-z_Q|]^{d+1}}|Q||x-y|.
\Ees
Notice $x\in (G_\lam)^c$ implies that $x\in (F_\lam)^c$ in Step 3. Therefore we see that
\Be\label{e:12lipb}
\big|A_{i,j}^{b}(x)-A_{i,j}^{b}(y)\big|\lc\lam^{\fr{r}{q_i}}|x-y|\sum_{Q\in\mathcal{Q}_{j,i}}\fr{l(Q)}{[l(Q)+|x-z_Q|]^{d+1}}|Q|\leq\lam^{\fr{r}{q_i}}|x-y|.
\Ee

Now we conclude the Lipschitz estimates related good functions \eqref{e:12lipg} and bad functions \eqref{e:12lipb} to get that for any $i=1,\cdots,n$, $x,y\in (G_\lam)^c$,
\Be\label{e:12qileqdlip}
|A_i(x)-A_i(y)|\lc\lam^{\fr{r}{q_i}}|x-y|.
\Ee

\vskip 0.24cm

\subsubsection*{Step 7: Estimate of $E_\lam$} We return to give an estimate of $E_\lam$. Split $f$ into two parts $f=f_1+f_2$ where $f_1(x)=f(x)\chi_{(G_\lam)^c}(x)$ and $f_2(x)=f(x)\chi_{G_\lam}(x)$. By the Lipschitz estimate in \eqref{e:12qileqdlip}, when restricted on $(G_\lam)^c$, $A_i$ is a Lipschitz function with $\|\nabla A_i\|_{L^\infty((G_\lam)^c)}\leq\lam^{\fr{r}{q_i}}$ for $i=1,\cdots,n$. Let $\tilde{A}_i$ stand for the Lipschitz extension of $A_i$ from $(G_\lam)^c$ to $\R^d$ (see \cite[page 174, Theorem 3]{Ste70}) so that
for each $i=1,\cdots,n$,
\Be\label{e:12qileqde}
\begin{split}
\tilde{A}_i(y)&=A_i(y)\ \ \text{if} \ y\in (G_\lam)^c;\\
\big|\tilde{A}_i(x)-\tilde{A}_i(y)\big|&\leq\lam^{\fr{r}{q_i}}|x-y|\ \ \text{for all}\ x,y\in\R^d.
\end{split}
\Ee

As we have done in \eqref{e:12spmu} and \eqref{e:12qbigd} in the proof of Proposition \ref{p:12qibigd}, we may reduce the proof to the following estimate
\Bes
m\big(\{x\in (10G_\lam)^c:|\mathcal{C}[\nabla A_1,\cdots,\nabla A_n,f_2](x)|>\lambda/2\}\big)\lc\lam^{-r}.
\Ees

\subsubsection*{Step 8: Estimate of $\mathcal{C}[\nabla A_1,\cdots,\nabla A_n, f_2](x)$} Recall $\N_i^j=\{i,i+1,\cdots,j\}$ and our construction of $G_\lam$, $y_k$, $Q_k$ and $Q_k^*$ in the paragraph above \eqref{e:12whitney2}. Then we can write $f_2=\sum_k f\chi_{Q_k}$. Therefore we may get
$$\mathcal{C}[\nabla A_1,\cdots,\nabla A_n, f_2](x)=\sum_k\mathcal{C}[\nabla A_1,\cdots,\nabla A_n, f\chi_{Q_k}](x).$$
Below we study $\prod_{i=1}^n\fr{{A}_i(x)-A_i(y)}{|x-y|}$. We will separate it into several terms and then give an estimate for each term.  Write
\Bes
\begin{split}
&\ \ \ \ \prod_{i=1}^n\fr{{A}_i(x)-A_i(y)}{|x-y|}\\&=\prod_{i=1}^n\Big(\fr{\tilde{A}_i(x)-\tilde{A}_i(y)}{|x-y|}+\fr{\tilde{A}_i(y)-\tilde{A}_i(y_k)}{|x-y|}+\fr{A_i(y_k)-A_i(y)}{|x-y|}\Big)\\
&=\sum\Big(\prod_{i\in N_1}\fr{\tilde{A}_i(x)-\tilde{A}_i(y)}{|x-y|}\Big)\Big(\prod_{i\in N_2}\fr{\tilde{A}_i(y)-\tilde{A}_i(y_k)}{|x-y|}\Big)\Big(\prod_{i\in N_3}\fr{{A}_i(y_k)-{A}_i(y)}{|x-y|}\Big)\\
&=I(x,y)+II(x,y,y_k)+III(x,y,y_k),
\end{split}
\Ees
where in the third equality we divide $\N_{1}^n=N_1\cup N_2\cup N_3$ with $N_1$, $N_2$, $N_3$ non intersecting each other; and $I(x,y)$, $II(x,y,y_k)$, $III(x,y,y_k)$ are defined as follows
\Be\label{e:12axyqbigd2}
\begin{split}
I=&\prod_{i=1}^n\fr{\tilde{A}_i(x)-\tilde{A}_i(y)}{|x-y|},\\
II=&\sum_{N_1\subsetneq\N_{1}^n\atop N_3=\emptyset}\Big(\prod_{i\in N_1}\fr{\tilde{A}_i(x)-\tilde{A}_i(y)}{|x-y|}\Big)\Big(\prod_{i\in N_2}\fr{\tilde{A}_i(y)-\tilde{A}_i(y_k)}{|x-y|}\Big),\\
III=&\sum_{N_1\subsetneq\N_{1}^n\atop N_3\neq\emptyset}\Big(\prod_{i\in N_1}\fr{\tilde{A}_i(x)-\tilde{A}_i(y)}{|x-y|}\Big)
\Big(\prod_{i\in N_2}\fr{\tilde{A}_i(y)-\tilde{A}_i(y_k)}{|x-y|}\Big)\\
&\qquad\qquad\qquad\times\Big(\prod_{i\in N_3}\fr{{A}_i(y_k)-{A}_i(y)}{|x-y|}\Big).
\end{split}
\Ee
By the above decomposition, we in fact divide $\mathcal{C}[\nabla A_1,\cdots,\nabla A_n, f\chi_{Q_k}](x)$ into $3^n$ terms. We separate these terms into three parts according $I$, $II$ and $III$.

\subsubsection*{Step 9: Estimate of $\mathcal{C}[\cdots,\cdot]$ related to $I$.}  This estimate is similar to the term related to $I$ in the proof of Proposition \ref{p:12qibigd}. In fact, in this case there is only one term $\C[\nabla\tilde{A}_1,\cdots,\nabla\tilde{A}_n,f_2]$. Then by Proposition \ref{p:12strongr} ($1<p<\infty$) and Proposition \ref{p:12inftyweak} ($p=1$), we get
\Bes
\begin{split}
m\big(\big\{x\in& (10G_\lam)^c: |\mathcal{C}[\nabla \tilde{A}_1,\cdots,\nabla\tilde{A}_n,f_2](x)|>{\lam}/{2}\big\}\big)\\
&\lc \lam^{-p}\Big(\prod_{i=1}^n\|\nabla\tilde{A}_i\|^p_{L^\infty(\R^d)}\Big)\|f_2\|^p_{L^p(\R^d)}
\lc\lam^{-p+p\sum_{i=1}^n\fr{r}{q_i}}=\lam^{-r}.
\end{split}
\Ees
If $p=\infty$, the above argument may not work.

\subsubsection*{Step 10: Estimate of $\mathcal{C}[\cdots,\cdot]$ related to $II$.} The proof of this part is similar to the estimate related to $II$ in Proposition \ref{p:12qibigd}.
It suffices to consider one term $\mathcal{C}[\cdots,\cdot]$ related to $II$ in which $N_1$ is a proper subset of $\N_1^n$.
In this case, without loss of generality, we may assume $N_1=\{1,\cdots,l\}$, $N_2=\{l+1,\cdots,n\}$ with $0\leq l<n$. Here when $l=0$, it means that $N_1=\emptyset$. With these notation, it is easy to see that $N_1$ is a proper subset of $\N_1^n$.
By a slight abuse of notation, we still use $II(x,y,y_k)$ to represent one term related to $N_1$ and $N_2$ in \eqref{e:12axyqbigd2} and use $H_{II}(x)$ to represent $\mathcal{C}[\cdots,\cdot]$ related to ${II}(x,y,y_k)$, i.e.
$$H_{II}(x)=\sum_k\int_{Q_k}K(x-y)II(x,y,y_k)f(y)dy.$$
Notice that $\tilde{A}_i$ is a Lipschitz function with bound $\lam^{\fr{r}{q_i}}$ by \eqref{e:12qileqde} for all $i=1,\cdots,n$. Then we get
\Bes
\begin{split}
|II(x,y,y_k)|\lc\lam^{\sum_{i=1}^n\fr{r}{q_i}}\fr{|y-y_k|^{n-l}}{|x-y|^{n-l}}.
\end{split}
\Ees
Since we only consider $x\in(10G_\lam)^c$, then by \eqref{e:12whitney2}, we get
\Be\label{e:12xygeqQ}
|x-y|\geq 2l(Q_k)\approx|y-y_k|\ \ \text{for any $y\in Q_k$}.
\Ee
Therefore utilizing \eqref{e:12kb} and the above estimate, we get
\Bes
\begin{split}
|H_{II}(x)|&\leq\sum_k\int_{Q_k}|K(x-y)|\cdot|II(x,y,y_k)|\cdot|f(y)|dy\\
&\lc\lam^{\sum_{i=1}^n\fr{r}{q_i}}\sum_k\int_{Q_k}\fr{l(Q_k)^{n-l}}{[l(Q_k)+|x-y|]^{d+n-l}}|f(y)|dy\\
&=\lam^{\sum_{i=1}^n\fr{r}{q_i}}T_{n-l}f(x),
\end{split}
\Ees
where the operator $T_{n-l}$ is defined in Lemma \ref{l:12disq1infty} with $|\Om|\equiv1$.
Now applying the Chebyshev inequality and the above estimate, and using Lemma \ref{l:12disq1infty} since $n-l\geq1$, we finally get
\Bes
\begin{split}
m(\{x\in(G_\lam)^c: |H_{II}(x)|>\lam\})&\leq\lam^{-p+\sum_{i=1}^n\fr{rp}{q_i}}\int_{(10G_\lam)^c}|T_{n-l}f(x)|^pdx\\
&\lc\lam^{-r}\|f\|^p_{L^p(\R^d)}.
\end{split}
\Ees
Hence we complete the proof related to $II$.

\subsubsection*{Step 11: Estimate of $\mathcal{C}[\cdots,\cdot]$ related to $III$.}
It suffices to consider one term $\mathcal{C}[\cdots,\cdot]$ related to $III$ in which $N_1$ is a proper subset of $\N_1^n$ and $N_3$ is a nonempty set.
In this case, without loss of generality, we may assume $N_1=\{1,\cdots,l\}$, $N_2=\{l+1,\cdots,m\}$ and $N_3=\{m+1,\cdots, n\}$ with $0\leq l\leq m<n$. Here when $l=0$, it means that $N_1=\emptyset$; when $l=m$, $N_2=\emptyset$. With these notation, one can easily see that $N_1$ is a proper subset of $\N_1^n$ and $N_3$ is a nonempty set.
By a slight abuse of notation, we still use $III(x,y,y_k)$ to represent one term related to $N_1$, $N_2$ and $N_3$ in \eqref{e:12axyqbigd2} and use $H_{III}(x)$ to represent $\mathcal{C}[\cdots,\cdot]$ related to ${III}(x,y,y_k)$, i.e.
$$H_{III}(x)=\sum_k\int_{Q_k}K(x-y)III(x,y,y_k)f(y)dy.$$

Recall in Step 2, we set $\fr{1}{s_i}=\fr{1}{q_i}-\fr{1}{d}$ for all $i=1,\cdots,n$. We also set $\fr{1}{q}=\Big(\sum_{i=m+1}^n\fr{1}{s_i}\Big)+\fr{1}{p}.$
Since $r\geq d/(d+n)$ and $\fr{1}{r}=\Big(\sum_{i=1}^n\fr{1}{q_i}\Big)+\fr{1}{p}$, by some elementary calculation, one may get $1\leq q\leq\infty$, this will be crucial when we use Lemma \ref{l:12disq1infty}. This is the place where we use the condition $r\geq d/(d+n)$. With the above fact and $\tilde{A}_i$ is a Lipschitz function with bound $\lam^{r/q_i}$ for $i=1,\cdots,m$, we get
\Bes
\begin{split}
|III(x,y,y_k)|\lc\lam^{\sum_{i=1}^m\fr{r}{q_i}}\fr{(l(Q_k))^{n-l}}{|x-y|^{n-l}}\prod_{i=m+1}^n\fr{|A_i(y_k)-A_i(y)|}{l(Q_k)}.
\end{split}
\Ees
Applying \eqref{e:12kb} and the above estimate with \eqref{e:12xygeqQ}, we get
\Bes
\begin{split}
|H_{III}(x)|&\leq\sum_k\int_{Q_k}|K(x-y)|\cdot|III(x,y,y_k)|\cdot|f(y)|dy\\
&\lc\lam^{\sum_{i=1}^m\fr{r}{q_i}}\sum_k\int_{Q_k}\fr{l(Q_k)^{n-l}}{[l(Q_k)+|x-y|]^{d+n-l}}h_{m,n}(y)dy\\
&\lc\lam^{\sum_{i=1}^m\fr{r}{q_i}}T_{n-l}\big(h_{m,n}\big)(x)
\end{split}
\Ees
where the function $h_{m,n}(y)=:\sum_{Q_k}\prod_{i=m+1}^n\Big(\fr{|A_i(y_k)-A_i(y)|}{l(Q_k)}\Big)\chi_{Q_k}|f|(y).$

Applying the Chebyshev inequality and the above estimate of $H_{III}$, utilizing Lemma \ref{l:12disq1infty} with $|\Om|\equiv1$, we then get
\Be\label{e:12Glam3qleqd}
\begin{split}
&m(\{x\in(10G_\lam)^c: |H_{III}(x)|>\lam\})\\
&\leq\lam^{-q+\big(\sum_{i=1}^m\fr{rq}{q_i}\big)}\int_{(10G_\lam)^c}|T_{n-l}(h_{m,n})(x)|^qdx\\
&\lc\lam^{-q+\big(\sum_{i=1}^m\fr{rq}{q_i}\big)}\|h_{m,n}\|^q_{L^q(\R^d)},
\end{split}
\Ee
since $1\leq q\leq\infty$ and $n-l\geq1$. Below we give an estimate of $\|h_{m,n}\|^q_{L^q(\R^d)}$. We write
\Be\label{e:12hmn}
\begin{split}
\|h_{m,n}\|^q_{L^q(\R^d)}&=\sum_{Q_k}\int_{Q_k}\prod_{i=m+1}^n\Big[\fr{|A_i(y_k)-A_i(y)|}{l(Q_k)}\Big]^q|f(y)|^qdy\\
&\leq\sum_{Q_k}\prod_{i=m+1}^n\Big[\int_{Q_k}\Big(\fr{|A_i(y_k)-A_i(y)|}{l(Q_k)}\Big)^{s_i}dy\Big]^{\fr{q}{s_i}}\Big(\int_{Q_k}|f(y)|^pdy\Big)^{\fr{q}{p}}\\
&\lc\sum_{Q_k}\prod_{i=m+1}^n\Big[\int_{Q^*_k}\Big(\fr{|A_i(y_k)-A_i(y)|}{l(Q^*_k)}\Big)^{s_i}dy\Big]^{\fr{q}{s_i}}\Big(\int_{Q_k}|f(y)|^pdy\Big)^{\fr{q}{p}}\\
&\lc\sum_{Q_k}\Big[\prod_{i=m+1}^n\mathfrak{M}_{s_i}(\nabla A_i)(y_k)^{q}|Q_k|^{\fr{q}{s_i}}\Big]\Big(\int_{Q_k}|f(y)|^pdy\Big)^{\fr{q}{p}},
\end{split}
\Ee
where the second inequality just follows from the H\"older inequality and in the third inequality we use the fact $Q_k\subset Q_k^*$, $y_k$ is the center of $Q^*_k$ and $l(Q^*_k)\approx l(Q_k)$. Notice that $y_k$ lies in the $(G_\lam)^c$, i.e. $y_k\in (D_{i,\lam})^c$. Therefore by the construction of $D_{i,\lam}$ in Step 2, we get
\Bes
\mathfrak{M}_{s_i}(\nabla A_i)(y_k)\leq\lam^{\fr{r}{q_i}}, \ \text{for $i=m+1,\cdots,n$.}
\Ees
Applying the above inequality, the H\"older inequality again and (cz-iii) in Step 1, we get
\Bes
\begin{split}
\|h_{m,n}\|^q_{L^q(\R^d)}&\lc\lam^{\sum_{i=m+1}^n\fr{qr}{q_i}}\Big(\sum_{Q_k}|Q_k|\Big)^{\sum_{i=m+1}^n\fr{q}{s_i}}\|f\|^q_{L^p(\R^d)}\\
&\lc\lam^{\sum_{i=m+1}^n\fr{qr}{q_i}}|G_\lam|^{\sum_{i=m+1}^n\fr{q}{s_i}}\lc\lam^{\sum_{i=m+1}^n\Big(\fr{qr}{q_i}-\fr{qr}{s_i}\Big)}.
\end{split}
\Ees
Submitting the above estimate into \eqref{e:12Glam3qleqd} with some elementary calculations, we finally get
\Bes
\begin{split}
m(\{x\in(10G_\lam)^c: |H_{III}(x)|>\lam\})\lc\lam^{-q+\big(\sum_{i=1}^n\fr{rq}{q_i}\big)-\big(\sum_{i=m+1}^n\fr{qr}{s_i}\big)}\lc\lam^{-r},
\end{split}
\Ees
which is the required bound. Hence we complete the proof. Notice that this argument for the term $III$ also works in the case $p=\infty$.
\end{proof}
\begin{prop}\label{e:12qleqgeqinfty}
Let $ d/(d+n)\leq r\leq1$, $1\leq q_1,\cdots,q_n< d$, $p=\infty$. Then
the weak type estimate \eqref{e:12muweak} holds.
\end{prop}
\begin{proof}
The proof is quite similar to that of Proposition \ref{p:12qleqdall}. One may follow the four different arguments that we deal with $E_\lam$, $f_1$, $I$ and $II$ in the proof of Proposition \ref{p:12qleqinfty}.
The proof of the term $III$ is similar to Step 11 in the proof of Proposition \ref{p:12qleqdall}. We omit the details of the proof here.
\end{proof}
\vskip0.24cm

\subsection{Case: $d/(d+n)\leq r\leq 1$, some $q_i\geq d$ and some $q_i< d$}\label{s:1225}\quad
\vskip0.24cm
In this subsection, we consider the most complicated case: $d/(d+n)\leq r\leq 1$ with some $q_i\geq d$ and some $q_i< d$. After the warm-up of the case all $q_i\geq d$ in Subsection \ref{s:1223} and all $q_i<d$ in Subsection \ref{s:1224}, the strategy  here is quite clear that we will put the two arguments in Subsections \ref{s:1223} and \ref{s:1224} together.
Without loss of generality, we may suppose that $d\leq q_1,\cdots,q_l\leq\infty$ and $1\leq q_{l+1},\cdots,q_n< d$ with $1\leq l<n$. Also we may assume that
$q_{1}=\cdots=q_{k}=d$ and $d<q_{k+1},\cdots,q_l\leq\infty$  with $0\leq k\leq l$. When $k=0$, we mean that there is no indice in $q_1,\cdots,q_l$ equals to $d$, i.e. $d<q_1,\cdots,q_l\leq\infty$; when $k=l$, we mean that $q_1=\cdots=q_l=d$. Since the proof of $p=\infty$ is a little different from that of $1\leq p<\infty$, we will give two propositions here. We first consider $1\leq p<\infty$.

\begin{prop}\label{p:12qgeqleqd}
Let $d/(d+n)\leq r\leq 1$, $q_{1}=\cdots=q_{k}=d$, $d< q_{k+1},\cdots,q_l\leq\infty$ and $1\leq q_{l+1},\cdots,q_n< d$ with $0\leq k\leq l$ and $1\leq l<n$, $1\leq p<\infty$. Then
\Be\label{e:12weakqbg}
\begin{split}
\|\C[\nabla A_1,&\cdots,\nabla A_n, f]\|_{L^{r,\infty}(\R^d)}\\
&\lc\Big(\prod_{i=1}^k\|\nabla A_i\|_{L^{d,1}(\R^d)}\Big)\Big(\prod_{i=k+1}^n\|\nabla A_i\|_{L^{q_i}(\R^d)}\Big)\|f\|_{L^p(\R^d)}.
\end{split}
\Ee
\end{prop}
\begin{proof}
The proof of this proposition is involved with the idea that we have done in the proof of Proposition \ref{p:12qibigd} for $q_i\geq d$ and Proposition \ref{p:12qleqdall} for $1\leq q_i<d$.
We will combine these two arguments in the proof of Proposition \ref{p:12qibigd} and Proposition \ref{p:12qleqdall}. One will see below that part of discussions have been appeared in the previous proposition.
So we shall be brief and only indicate necessary differences.

Now we start our proof. Our main goal is to prove that for any $\lam>0$, the following inequality holds
\Bes
\begin{split}
\lam^{r}m(\{&x\in\R^d:|\C[\nabla A_1,\cdots,\nabla A_n, f](x)|>\lam\})\\
&\lc\Big(\prod_{i=1}^k\|\nabla A_i\|^r_{L^{d,1}(\R^d)}\Big)\Big(\prod_{i=k+1}^n\|\nabla A_i\|^r_{L^{q_i}(\R^d)}\Big)\|f\|^r_{L^p(\R^d)}.
\end{split}
\Ees

Fix $\lam>0$. Recall $E_\lam$ defined in \eqref{e:12elam}.
By rescaling as showed in the proof of Proposition \ref{p:12qibigd} or Proposition \ref{p:12qleqdall}, it suffices to show $|E_\lam|\lc\lam^{-r}$. The main idea is to construct some {\it exceptional set\/} such that the measure of {\it exceptional set\/} is bounded by $\lam^{-r}$, which is our required estimate. At the same time on the complementary set of {\it exceptional set\/} these functions $A_i$ should be Lipschitz functions with bound $\lam^{\fr{r}{q_i}}$ for each $i=1,\cdots,n$. If $d\leq q_i<\infty$, the construction of {\it exceptional set\/} is similar to that of Proposition \ref{p:12qibigd}. And if $1\leq q_i<d$, the construction of {\it exceptional set\/} is similar to that of Proposition \ref{p:12qleqdall}. As we have done in Proposition \ref{p:12qibigd}, we only need to consider that all $q_{k+1},\cdots, q_l<\infty$. Below we begin our constructions of some {\it exceptional set\/}.

\subsubsection*{Step 1: Exceptional set $J_\lam$} Define the {\it exceptional set} {for $i=1,\cdots,l$}
\Bes
\begin{split}
J_{i,\lam}=\big\{x\in\R^d:\M (\nabla A_i)(x)>\lam^{\fr{r}{q_i}}\big\};\ \ J_\lam=\cup_{i=1}^l J_{i,\lam}.
\end{split}
\Ees

\subsubsection*{{Step 2: Calder\'on-Zygmund decomposition.}}
For each $|\pari_jA_i|^{q_i}\in L^1(\R^d)$ with $j=1,\cdots, d$ and $i=l+1,\cdots,n$, making a Calder\'on-Zygmund decomposition at level $\lam^r$ as we have done in the proof of Proposition \ref{p:12qleqdall},
one may get the properties of $g_{j,i}$, $b_{j,i}$, $E_{j,i}$, $A^g_{i,j}$, $A^b_{i,j}$ similarly.
Set the \emph{exceptional set}$B_\lam=\cup_{i=l+1}^n\cup_{j=1}^dE_{j,i}$.

\subsubsection*{Step 3: Exceptional set $D_\lam$.} Set $\fr{1}{s_i}=\fr{1}{q_i}-\fr{1}{d}$ for $i=l+1,\cdots,n$. Define the following \emph{exceptional set}
\Bes
D_{i,\lam}=\Big\{x\in\R^d:\mathfrak{M}_{s_i}(\nabla A_i)(x)>\lam^{\fr{r}{q_i}}\Big\},\ \ D_\lam=\cup_{i=l+1}^nD_{i,\lam}.
\Ees

\subsubsection*{Step 4: Exceptional set $F_\lam$.} For each $j=1,\cdots,d$, $i=l+1,\cdots,n$, we define the functions $\Delta^{j,i}(x)$ as those in the proof of Proposition \ref{p:12qleqdall}.
Define another \emph{exceptional set}
\Bes
F_{j,i,\lam}=\{x\in\R^d:\Delta^{j,i}(x)>1\},\ \ F_{\lam}=\cup_{j=1}^d\cup_{i=l+1}^n F_{j,i,\lam}.
\Ees

\subsubsection*{Step 5: Exceptional set $H_\lam$.} We define the {\it exceptional set\/} for $i=l+1,\cdots,n$, $j=1,\cdots,d$,
\Bes
H_{i,j,\lam}=\{x\in\R^d:\M(\nabla A^g_{i,j})(x)>\lam^{r/q_i}\},\ \ H_\lam=\cup_{i=l+1}^n\cup_{j=1}^d H_{i,j,\lam}.
\Ees

\subsubsection*{Step 6: Final exceptional set $G_\lam$} Based on the construction of $J_\lam, B_\lam, D_\lam, F_\lam, H_\lam$ in Steps 1-5,
we choose an open set $G_\lam$ which satisfies the following conditions:
\begin{enumerate}[(1).]
\item\quad $\big(10J_\lam\cup 10B_\lam\cup 10D_\lam\cup 10F_\lam\cup 10H_\lam\big)\subset G_\lam$;
\item\quad $m(G_\lam)\leq(20)^d(|J_\lam|+|B_\lam|+|D_\lam|+|F_\lam|+|H_\lam|)$.
\end{enumerate}
As showed in the proof of Proposition \ref{p:12qibigd} and Proposition \ref{p:12qleqdall}, one may get that the measures of $J_\lam$, $B_{\lam}$, $D_\lam$, $F_\lam$ and $H_\lam$ are bounded by $\lam^{-r}$. So we see that $m(G_\lam)\lc\lam^{-r}$. Next making a Whitney decomposition of $G_\lam$, we may get a family of disjoint dyadic cubes $\{Q_k\}_k$ and then we construct a larger cube $Q_k^*$ so that $Q_k\subset Q_k^*$, $Q_k^*$ is centered at $y_k$ and $y_k\in (G_\lam)^c$, $l(Q_k^*)\approx l(Q_k)$.
By the construction of $Q_k^*$ and $y_k$, one may get
\Be\label{e:12whitney3}
dist(y_k,Q_k)\approx l(Q_k).
\Ee
In the following we will show that these functions $A_i$ are Lipschitz functions on the complementary set of $G_\lam$.

\subsubsection*{Step 7: Lipschitz estimates of $A_i$ on $(G_\lam)^c$} Choose any $x,y\in(G_\lam)^c$. By the exceptional set $J_\lam$ constructed in Step 1, we see that for $i=1,\cdots,l$
\Be\label{e:12lipgeqd}
|A_i(x)-A_i(y)|\leq\lam^{\fr{r}{q_i}}|x-y|.
\Ee
Below we consider $i=l+1,\cdots,n$.
By the Calder\'on-Zygmund decomposition in Step 2, it suffices to show that $A_{i,j}^g$ and $A_{i,j}^b$ satisfy Lipschitz estimates on $\big(G_\lam\big)^c$ for each $i=l+1,\cdots,n$ and $j=1,\cdots,d$. Firstly, one may easily see that $A_{i,j}^g$ satisfies Lipschitz estimates by the construction of $H_\lam$ in Step 5. In fact, $x,y\in (G_\lam)^c$ implies that $x,y\in H^c_\lam$, we get for $i=l+1,\cdots,n$, $j=1,\cdots,d$,
\Be\label{e:12lipg3}
|A_{i,j}^g(x)-A_{i,j}^g(y)|\leq\lam^{\fr{r}{q_i}}|x-y|.
\Ee

While considering $A_{i,j}^b$, we see that by using the similar method that we prove \eqref{e:12lipb} in Proposition \ref{p:12qleqdall}, we get for $i=l+1,\cdots,n$, $j=1,\cdots,d$,
\Be\label{e:12lipb3}
\big|A_{i,j}^{b}(x)-A_{i,j}^{b}(y)\big|\lc\lam^{\fr{r}{q_i}}|x-y|.
\Ee

Therefore we conclude the Lipschitz estimates in \eqref{e:12lipgeqd} for $i=1,\cdots,l$, good function \eqref{e:12lipg3} and bad function \eqref{e:12lipb3} for $i=l+1,\cdots,n$, to get that for any $i=1,\cdots,n$, $x,y\in (G_\lam)^c$,
\Be\label{e:12qileqdlip3}
|A_i(x)-A_i(y)|\leq\lam^{\fr{r}{q_i}}|x-y|.
\Ee
\vskip 0.24cm

\subsubsection*{Step 8: Estimate of $E_\lam$} As we have done in the proof of Proposition \ref{p:12qibigd}, we may split $f=f_1+f_2$. Following \eqref{e:12spmu} and \eqref{e:12qbigd}, we may reduce the estimate of $E_\lam$ to the following inequality
\Bes
m\big(\{x\in (10G_\lam)^c:|\mathcal{C}[\nabla A_1,\cdots,\nabla A_n,f_2](x)|>\lambda/2\}\big)\lc\lam^{-r}.
\Ees

\subsubsection*{Step 9: Estimate of $\mathcal{C}[\nabla A_1,\cdots,\nabla A_n, f_2](x)$}  Recall $\N_i^j=\{i,i+1,\cdots,j\}$ and our construction of $G_\lam$, $y_k$, $Q_k$ and $Q_k^*$ in the paragraph above \eqref{e:12whitney3}. Then we can write $f_2=\sum_k f\chi_{Q_k}$. Therefore we may get
$$\mathcal{C}[\nabla A_1,\cdots,\nabla A_n, f_2](x)=\sum_k\mathcal{C}[\nabla A_1,\cdots,\nabla A_n, f\chi_{Q_k}](x).$$
Below we study $\prod_{i=1}^n\fr{{A}_i(x)-A_i(y)}{|x-y|}$. We will separate it into several terms and then give an estimate for each term.  Write
\Bes \prod_{i=1}^n\fr{{A}_i(x)-A_i(y)}{|x-y|}=I(x,y)+II(x,y,y_k)+III(x,y,y_k)+IV(x,y,y_k),
\Ees
where $I(x,y)$, $II(x,y,y_k)$, $III(x,y,y_k)$ and $IV(x,y,y_k)$ are defined as follows
\Be\label{e:12axyqbigd3}
\begin{split}
I=&\prod_{i=1}^n\fr{\tilde{A}_i(x)-\tilde{A}_i(y)}{|x-y|},\\
II=&\sum_{N_1\subsetneq\N_{1}^n\atop N_3=\emptyset}\Big(\prod_{i\in N_1}\fr{\tilde{A}_i(x)-\tilde{A}_i(y)}{|x-y|}\Big)\Big(\prod_{i\in N_2}\fr{\tilde{A}_i(y)-\tilde{A}_i(y_k)}{|x-y|}\Big),\\
III=&\sum_{N_1\subsetneq\N_{1}^n\atop N_3\neq\emptyset, N_3\subset\{1,\cdots,l\}}\Big(\prod_{i\in N_1}\fr{\tilde{A}_i(x)-\tilde{A}_i(y)}{|x-y|}\Big)
\Big(\prod_{i\in N_2}\fr{\tilde{A}_i(y)-\tilde{A}_i(y_k)}{|x-y|}\Big)\\
&\qquad\qquad\qquad\times\Big(\prod_{i\in N_3}\fr{{A}_i(y_k)-{A}_i(y)}{|x-y|}\Big),\\
IV=&\sum_{N_1\subsetneq\N_{1}^n\atop N_3 \neq\emptyset, N_3\cap\{l+1,\cdots,n\}\neq\emptyset}\Big(\prod_{i\in N_1}\fr{\tilde{A}_i(x)-\tilde{A}_i(y)}{|x-y|}\Big)
\Big(\prod_{i\in N_2}\fr{\tilde{A}_i(y)-\tilde{A}_i(y_k)}{|x-y|}\Big)\\
&\qquad\qquad\qquad\times\Big(\prod_{i\in N_3}\fr{{A}_i(y_k)-{A}_i(y)}{|x-y|}\Big),
\end{split}
\Ee
here $\N_{1}^n=N_1\cup N_2\cup N_3$ with $N_1$, $N_2$, $N_3$ non intersecting each other. By the above decomposition, we in fact divide $\mathcal{C}[\nabla A_1,\cdots,\nabla A_n, f\chi_{Q_k}](x)$ into $3^n$ terms. We separate these terms into four parts according $I$, $II$, $III$ and $IV$.

\subsubsection*{Step 10: Estimate of $\mathcal{C}[\cdots,\cdot]$ related to $I$.}  Since $I$ is the same as $I$ term in the proof of Proposition \ref{p:12qibigd}, so this estimate is similar to that there. We omit the proof here.

\subsubsection*{Step 11: Estimate of $\mathcal{C}[\cdots,\cdot]$ related to $II$.} This estimate is similar to the term related to $II$ in the proof of Proposition \ref{p:12qleqdall}. So we omit the proof.

\subsubsection*{Step 12: Estimate of $\mathcal{C}[\cdots,\cdot]$ related to $III$.}
It suffices to consider one term $\mathcal{C}[\cdots,\cdot]$ related to $III$ in which $N_1$ is a proper subset of $\N_1^n$ and $N_3$ is a nonempty subset of $\{1,\cdots,l\}$. By the condition in this proposition, for any $i\in N_3$, $d\leq q_i<\infty$. Thus $\nabla A_i\in L^{q_i}(\R^d)$ (or $L^{d,1}(\R^d)$ if $q_i=d$) with $d\leq q_i<\infty$. Therefore the estimates in $N_3$ will be straightforward since
\Be\label{e:12mbablaiN3}
\fr{|A_i(y_k)-A_i(y)|}{|y-y_k|}\leq\M(\nabla A_i)(y_k)\leq\lam^{\fr{r}{q_i}}, \ \text{for $i\in N_3$.}
\Ee
Once we give the above estimate in $N_3$, the rest terms related to $N_1$ and $N_2$ can be dealt as the same way to those related to $II$. For the rest of the proof, one can follow the term related to $II$ in the proof of Proposition \ref{p:12qleqdall}.

\subsubsection*{Step 13: Estimate of $\mathcal{C}[\cdots,\cdot]$ related to $IV$.}
It suffices to consider one term $\mathcal{C}[\cdots,\cdot]$ related to $IV$ in which $N_1$ is a proper subset of $\N_1^n$ and $N_3$ is a nonempty set with $N_3\cap\{l+1,\cdots,n\}\neq\emptyset$.
In this case, without loss of generality, we may assume $l+1,\cdots, v\in N_3$ with $l+1\leq v\leq n$ and $v+1,\cdots,n$ belongs to $N_1$ or $N_2$. So we may suppose that $N_3=\{\iota,\cdots,w, l+1,\cdots, v\}$ with $0\leq\iota\leq w\leq l$ and $N_1\neq\emptyset$. Set $u=\card (N_1)$. Then $n-u\geq1$. With these notation, it is easy to see that $N_3$ is a nonempty set with $N_3\cap\{l+1,\cdots,n\}\neq\emptyset$.
We use $IV(x,y,y_k)$ to represent one term related to $N_1$, $N_2$ and $N_3$ in \eqref{e:12axyqbigd3} and use $H_{IV}(x)$ to represent $\mathcal{C}[\cdots,\cdot]$ related to ${IV}(x,y,y_k)$, i.e.
$$H_{IV}(x)=\sum_k\int_{Q_k}K(x-y)IV(x,y,y_k)f(y)dy.$$

By our condition, $d\leq q_{l},\cdots,q_l\leq\infty$ and $1\leq q_{l+1},\cdots,q_n<d$. Recall in Step 3, we set $\fr{1}{s_i}=\fr{1}{q_i}-\fr{1}{d}$ for $i=l+1,\cdots,n$. We also set $\fr{1}{q}=\big(\sum_{i=l+1}^v\fr{1}{s_i}\big)+\fr{1}{p}.$
Since $r\geq d/(d+n)$ and $\fr{1}{r}=\big(\sum_{i=1}^n\fr{1}{q_i}\big)+\fr{1}{p}$, by some elementary calculation, one may get $1\leq q\leq\infty$. With the above fact and $\tilde{A}_i$ is a Lipschitz function with bound $\lam^{r/q_i}$ for $i\in N_1\cup N_2$, we get
\Bes
\begin{split}
|IV(x,y,y_k)|&\lc\lam^{(\sum_{i\in N_1\cup N_2})\fr{r}{q_i}}\fr{(l(Q_k))^{n-u}}{|x-y|^{n-u}}\Big(\prod_{i=\iota}^w\M(\nabla A_i)(y_k)\Big)\\
&\qquad\qquad\qquad\times\prod_{i=l+1}^v\fr{|A_i(y_k)-A_i(y)|}{l(Q_k)}\\
&\lc\lam^{(\sum_{i=1}^l+\sum_{i=v+1}^n)\fr{r}{q_i}}\fr{(l(Q_k))^{n-u}}{|x-y|^{n-u}}\prod_{i=l+1}^v\fr{|A_i(y_k)-A_i(y)|}{l(Q_k)}.
\end{split}
\Ees
Then inserting the above estimate of $IV$ into $H_{IV}$ with \eqref{e:12kb}, we get
\Bes
\begin{split}
|H_{IV}(x)|&\leq\sum_k\int_{Q_k}|K(x-y)|\cdot|IV(x,y,y_k)|\cdot|f(y)|dy\\
&\lc\lam^{(\sum_{i=1}^l+\sum_{i=v+1}^n)\fr{r}{q_i}}\sum_k\int_{Q_k}\fr{l(Q_k)^{n-u}}{[l(Q_k)+|x-y|]^{d+n-u}}h_{l,v}(y)dy\\
&=\lam^{(\sum_{i=1}^l+\sum_{i=v+1}^n)\fr{r}{q_i}}T_{n-u}\big(h_{l,v}\big)(x).
\end{split}
\Ees
Now the rest of the proof is similar to \eqref{e:12Glam3qleqd} in the proof of Proposition \ref{p:12qleqdall}. We omit the details here.
\end{proof}
\begin{prop}\label{p:12qgeqleqdinfty}
Let $d/(d+n)\leq r\leq 1$, $q_{1}=\cdots=q_{k}=d$, $d< q_{k+1},\cdots,q_l\leq\infty$ and $1\leq q_{l+1},\cdots,q_n< d$ with $0\leq k\leq l$ and $1\leq l<n$, $p=\infty$. Then the weak type estimate \eqref{e:12weakqbg} holds.
\end{prop}
\begin{proof}
The proof is similar to that of Proposition \ref{p:12qgeqleqd}, one may follow the idea in the proof of Proposition \ref{p:12qleqinfty} and Proposition \ref{e:12qleqgeqinfty}. We omit the details here.
\end{proof}
\vskip0.24cm
\subsection{Multilinear interpolation arguments}\label{s:1226}\quad
\vskip0.24cm
Notice that we have already proven all cases (ii) in Theorem \ref{t:12} by Propositions \ref{p:12inftyweak}, \ref{p:12qibigd}, \ref{p:12qleqinfty}, \ref{p:12qleqdall}, \ref{e:12qleqgeqinfty}, \ref{p:12qgeqleqd}, \ref{p:12qgeqleqdinfty}. And we also prove the case $1\leq r<\infty$ of (i) in Theorem \ref{t:12} by Proposition \ref{p:12strongr}. The rest part of (i) in Theorem \ref{t:12} follows from the standard multilinear interpolation.
In fact,  in the case $1\leq r<+\infty$, for all point $(\fr{1}{q_1},\cdots,\fr{1}{q_n},\fr{1}{p})$ in the polyhedron $\big(\sum_{i=1}^n\fr{1}{q_i}\big)+\fr{1}{p}=\fr{1}{r}$, we have the follow strong type estimate
\Bes
\|\C[\nabla A_1,\cdots,\nabla A_n, f]\|_{L^r(\R^d)}\lc\Big(\prod_{i=1}^n\|\nabla A_i\|_{L^{q_i}(\R^d)}\Big)\|f\|_{L^p(\R^d)}.
\Ees
In the case $r=\fr{d}{d+n}$, for all points $(\fr{1}{q_1},\cdots,\fr{1}{q_n},\fr{1}{p})$ in the plane $\big(\sum_{i=1}^n\fr{1}{q_i}\big)+\fr{1}{p}=\fr{d+n}{d}$, we have the weak type estimate
$$\|\mathcal{C}[\nabla A_1,\cdots,\nabla A_n,f]\|_{L^{\fr{d}{d+n},\infty}(\R^d)}\leq C \Big(\prod_{i=1}^n\|\nabla A_i\|_{L^{q_i}(\R^d)}\Big)\|f\|_{L^p(\R^d)}$$
where $L^{q_i}(\R^d)$ in the above inequality should be replaced by $L^{d,1}(\R^d)$ if $q_i=d$ for some $i=1,\cdots,n$. Notice that in the rest part of (i) in Theorem \ref{t:12}, we consider $\fr{d}{d+n}<r<1$, $1<q_i\leq\infty$ $(i=1,\cdots,n)$ and $1<p\leq\infty$, thus the point $(\fr{1}{q_1},\cdots,\fr{1}{q_n},\fr{1}{p})$ lies in the interior of the polyhedron between the polyhedron $\big(\sum_{i=1}^n\fr{1}{q_i}\big)+\fr{1}{p}=\fr{1}{r}$ with $1\leq r<\infty$ and
plane $\big(\sum_{i=1}^n\fr{1}{q_i}\big)+\fr{1}{p}=\fr{d+n}{d}$. Then if we choose $n+1$ points in the polyhedron $\big(\sum_{i=1}^n\fr{1}{q_i}\big)+\fr{1}{p}=\fr{1}{r}$ with $1\leq r<\infty$ and one point in the plane $\big(\sum_{i=1}^n\fr{1}{q_i}\big)+\fr{1}{p}=\fr{d+n}{d}$, using the multilinear interpolation theorem (see \cite[Theorem 7.2.2]{Gra250}), we get all strong type estimate in (i). Therefore we complete the proof of (i) and (ii) in Theorem \ref{t:12}. Proof of (iii) in Theorem \ref{t:12} will be given in the next subsection.
\vskip0.24cm
\subsection{Examples}\label{s:1227}\quad

\begin{prop}\label{p:12exa}
If $0<r<\fr{d}{d+n}$, $1\leq q_i\leq\infty$ $(i=1,\cdots,n)$ and $1\leq p\leq\infty$, there exist functions $K$, $A_i$ ($i=1,\cdots,n$) and $f$ such that $K$ satisfies \eqref{e:12kb}, \eqref{e:12K_2}, \eqref{e:12kr}; $\|\nabla A_i\|_{L^{q_i}(\R^d)}<\infty$ for $i=1,\cdots,n$ and $\|f\|_{L^p(\R^d)}<\infty.$ But
\Bes
\begin{split}
\mathcal{C}[\nabla A_1,\cdots,\nabla A_n,f](x)=\infty\ \ \text{in a ball in $\R^d$.}
\end{split}
\Ees
\end{prop}
\begin{proof}
We may suppose that $n$ is an odd integer. Then we may choose $K$ as
$$K(x)=\fr{1}{|x|^d}, \ \ \text{for $x\in\R^d\setminus\{0\}$}.$$
It is easy to see that $K$ satisfies \eqref{e:12kb}, \eqref{e:12K_2} and \eqref{e:12kr}.

Next we choose $\alp_i\ (i=1,\cdots,n)$ and $\beta$ such that  $-1\leq\alp_i<1$ for $i=1,\cdots,n$, $\beta\geq0$ and $(\sum_{i=1}^n\alp_i)+\beta=d$.
For each $i=1,\cdots,n$, we choose $A_i$ as a $C^\infty$ function in $\R^d\setminus\{0\}$ such that
\Bes
A_i(x)=\begin{cases}
&{|x|^{-\alp_i}}\ \ \text{if $x\in\ Cone_{\rho}$;} \\
&0\ \ \qquad\text{if $x\notin\ Cone_{\rho,\eps}$},
\end{cases}
\Ees
where $Cone_{\rho}$ and $Cone_{\rho,\eps}$ are defined as follows
\Bes
\begin{split}
Cone_{\rho}&=\big\{x\in\R^d: \kappa\sum_{j=2}^dx_j^2<x_1^2,\ 0<x_1<\rho\big\},\\
Cone_{\rho,\eps}&=\big\{x\in\R^d: \kappa\sum_{j=2}^dx_j^2<(x_1+\eps)^2,\ -\eps<x_1<\rho+\eps\big\},
\end{split}
\Ees
here $\rho$, $\eps$ are fixed positive constants and $\kappa$ is a large positive constant.
Below we choose
\Bes
f(x)=\begin{cases}
&{|x|^{-\beta}}\ \ \text{if $x\in\ Cone_{\rho}$,} \\
&0\ \qquad\text{if $x\notin\ Cone_{\rho}$}.
\end{cases}
\Ees
By some elementary calculation, one can easily get that for $1\leq q_i<\infty$ and $1\leq p<\infty$
\Bes
\begin{split}
\|\nabla A_i\|_{L^{q_i}(\R^d)}<\infty\ \ \text{if}\ \ {1\leq q_i<\fr{d}{1+\alp_i}};\ \
\|f\|_{L^p(\R^d)}<\infty\ \ \text{if}\ \ 1\leq p<\fr{d}{\beta}.
\end{split}
\Ees
If $q_i=\infty$, we may choose $\alp_i=-1$, then we get $\|\nabla A_i\|_{L^\infty(\R^d)}<\infty$. If $p=\infty$, we may choose $\beta=0$, then $\|f\|_{L^\infty(\R^d)}<\infty$. Since $(\sum_{i=1}^n\alp_i)+\beta=d$, it is impossible that all $q_i$ and $p$ equal to $\infty$. Then by this choice of $q_i$ and $p$, we see that $0<r<d/{(d+n)}$.

Set $z_0=(-2\eps,0,\cdots,0)$. Let $x$ be a point in the small neighborhood of $z_0$ such that
$$|x-y|\leq{C}{|\rho+4\eps|}\ \ \text{for all $y\in Cone_{\rho}$}.$$
Then combining the choice of $A_i$  ($i=1,\cdots,n$) and $f$, and noticing that $n$ is a odd integer, we finally get
\Bes
\begin{split}
-\mathcal{C}[\nabla A_1,\cdots,\nabla A_n,f](x)&=\int_{Cone_{\rho}}K(x-y)\Big[\prod_{i=1}^n\fr{A_i(y)}{|x-y|}\Big]f(y)dy\\
&\geq\fr{C}{|\rho+4\eps|^{d+n}}\int_{Cone_{\rho}}\fr{dy}{|y|^{d}}=+\infty.
\end{split}
\Ees

\end{proof}
\vskip0.24cm
\section{Proof of Theorem \ref{t:12r}}\label{s:123}
\vskip0.24cm
In this section, we just outline the proof of Theorem \ref{t:12r} since it is similar to that of Theorem \ref{t:12}.
\begin{prop}\label{p:12strongrom}
Let $1\leq r<+\infty$, $1<q_i\leq\infty$, $i=1,\cdots,n$, $1<p\leq\infty$. Then
\Bes
\|\C_\Om[\nabla A_1,\cdots,\nabla A_n, f]\|_{L^r(\R^d)}\lc C_\Om\Big(\prod_{i=1}^n\|\nabla A_i\|_{L^{q_i}(\R^d)}\Big)\|f\|_{L^p(\R^d)}.
\Ees
\end{prop}
\begin{proof}
We use the method of rotation to prove our main result. The method is standard, so we will be brief. Applying the condition \eqref{e:12Omho} and making a change of variable $x-y=r\tet$, we get that
\Bes
\begin{split}
\C_{\Om,\eps}[\nabla A_1,\cdots,&\nabla A_n, f](x)=\int_{|x-y|>\eps}\fr{\Om(x-y)}{|x-y|^d}\Big(\prod_{i=1}^n\fr{A_{i}(x)-A_i(y)}{|x-y|}\Big)\cdot f(y)dy\\
&=\int_{\S^{d-1}}\Om(\tet)\Big[\int_{r>\eps}\fr{1}{r}\Big(\prod_{i=1}^n\fr{A_{i}(x)-A_i(x-r\tet)}{r}\Big)\cdot f(x-r\tet)dr\Big]d\si(\tet)\\
&=\fr{1}{2}\int_{\S^{d-1}}\Om(\tet)\Big[\int_{|r|>\eps}\fr{1}{r}\Big(\prod_{i=1}^n\fr{A_{i}(x)-A_i(x-r\tet)}{r}\Big)\cdot f(x-r\tet)dr\Big]d\si(\tet).
\end{split}
\Ees
For convenience, we set the integral in the square bracket as $\C^{1}_{\eps,\tet}[\nabla A_1,\cdots, \nabla A_n, f](x)$. Now applying the Minkowski inequality, making a change of variable $x=s\tet+z$ with $s\in\R$ and $z\in L(\tet)$ ($L(\tet)$ is the hyperplane which is perpendicular to $\tet$), utilizing Proposition \ref{p:12strongr} with dimension one and the H\"older inequality, we finally get that
\Bes
\begin{split}
\|\C_\Om&[\nabla A_1,\cdots,\nabla A_n, f]\|_{L^r(\R^d)}\\
&\lc\int_{\S^{d-1}}|\Om(\tet)|\Big[\int_{L(\tet)}\int_{\R}|\C_{\eps,\tet}^1[\nabla A_1,\cdots, \nabla A_n, f](s\tet+z)|^rdsdz\Big]^{\fr{1}{r}}d\si(\tet)\\
&\lc\int_{\S^{d-1}}|\Om(\tet)|\Big[\int_{L(\tet)}\Big(\prod_{i=1}^n\|\nabla A_i(\cdot\tet+z)\|^r_{L^{q_i}(\R)}\Big)\|f(\cdot\tet+z)\|^r_{L^p(\R)}dz\Big]^{\fr{1}{r}}d\si(\tet)\\
&\lc \|\Om\|_{L^1(\S^{d-1})}\Big(\prod_{i=1}^n\|\nabla A_i\|_{L^{q_i}(\R^d)}\Big)\|f\|_{L^p(\R^d)},
\end{split}
\Ees
which ends the proof.
\end{proof}
\begin{remark}\label{r:12cancelation}
Proposition \ref{p:12strongrom} may also hold if we consider that $\Om$ is homogenous of order zero and satisfies
\begin{equation}\label{e:12moment condc}
\Om(-\tet)=(-1)^{n}\Om(\tet);\ \ \int_{\S^{d-1}}\Om(\tet)\tet^\alpha d\tet=0,\quad \text{for all}\ \ \alpha\in \Z_+^d\ \ \text{with}\ \ |\alp|=n;
\end{equation}
and $\Om\in L\log^+L(\S^{d-1})$. In such a case, then method of rotation can not be applies directly. One may need to insert $I=\sum_{j=1}^dR_j^2$ into the kernel of Calder\'on commutator where $R_j$ is the Riesz transform. Since the commutator is a non convolution operator, there are some tail terms that should be dealt carefully.
We don't pursue these matters in the present paper. For those reader who are interested in this case, we refer to see the very detailed discussion by B. Bajsanski and R. Coifman \cite{BC67}, which may also work here.
\end{remark}
\begin{prop}\label{p:12inftyweakrom}
Let $r=1$, $q_1=\cdots=q_n=\infty$, $p=1$. Suppose $\Om\in L\log^+L(\S^{d-1})$. Then
\Bes
\|\C_\Om[\nabla A_1,\cdots,\nabla A_n, f]\|_{L^{1,\infty}(\R^d)}\lc C_\Om\Big(\prod_{i=1}^n\|\nabla A_i\|_{L^{\infty}(\R^d)}\Big)\|f\|_{L^1(\R^d)}.
\Ees
\end{prop}
\begin{proof}
When $q_1=\cdots=q_n=\infty$, $A_i$ is a Lipschitz function for $i=1,\cdots,n$. Fix all $A_i$. We may regard $\C_\Om[\nabla A_1,\cdots,\nabla A_n, f](x)$ as a linear function of $f$. Then the kernel
$$\fr{1}{|x-y|^d}\Big(\prod_{i=1}^n\fr{A_i(x)-A_i(y)}{|x-y|}\Big)$$
is a standard Calder\'on-Zygmund kernel (see \eg \cite[Page 211, Definition 4.1.2]{Gra250}). Then by the recent result of Ding and the author \cite[Theorem 1.1]{DL15a}, one immediately finish the proof of Proposition \ref{p:12inftyweakrom}.
\end{proof}

After establishing Proposition \ref{p:12strongrom} and Proposition \ref{p:12inftyweakrom}, one can get the rest of the proof of Theorem \ref{t:12r} by using the similar way in the proof of Theorem \ref{t:12}.
One can check the proof step-by-step in which the applications of Proposition \ref{p:12strongr} and Proposition \ref{p:12inftyweak} in Propositions \ref{p:12qleqinfty}, \ref{p:12qleqdall}, \ref{e:12qleqgeqinfty}, \ref{p:12qgeqleqd}, \ref{p:12qgeqleqdinfty} are just replaced by Proposition \ref{p:12strongrom} and Proposition \ref{p:12inftyweakrom}. There is only one thing that we should be careful.
When giving a explicit estimate of $K(x-y)$ in Propositions \ref{p:12qleqinfty}, \ref{p:12qleqdall}, \ref{e:12qleqgeqinfty}, \ref{p:12qgeqleqd}, \ref{p:12qgeqleqdinfty}, we just use the boundedness condition \eqref{e:12kb}: $|K(x)|\lc|x|^{-d}$,  and then apply Lemma \ref{l:12disq1infty} in a special case $|\Om|\equiv1$ to get the required bound. While in the rest of the proof of Theorem \ref{t:12r}, the above arguments are replaced by $|K(x)|\leq\fr{|\Om(x)|}{|x|^d}$ and apply Lemma \ref{l:12disq1infty} with a rough kernel $\Om$.
The verification of the details of this proof is omitted.
Finally, it is easy to see that those examples in Proposition \ref{p:12exa} also work here.

\vskip0.24cm

\subsection*{Acknowledgements}
The author would like to thank Andreas Seeger for suggesting to consider the Lorentz space $L^{d,1}(\R^d)$ endpoint estimate of the Mary Weiss maximal operator $\M$,
thank Xiaohua Yao for pointing out the $n$-th order commutator in Corollary \ref{c:12c} falls into the scope of Theorem \ref{t:12},
and finally thank the referees for their very careful reading and valuable suggestions.
\vskip1cm

\bibliographystyle{amsplain}

\end{document}